\documentclass[12pt, a4paper, headsepline]{scrartcl} 

\RequirePackage[l2tabu, orthodox]{nag}  

\typearea[5mm]{15}    

\usepackage[english]{babel}

\usepackage{scrpage2}
\usepackage{graphicx}                
\newcommand{\color}[2][{}]{}         

\usepackage{amssymb}
\usepackage{amsmath}
\usepackage{amsthm}

\usepackage{ifthen}       
\usepackage{scrtime}      

\usepackage{mathabx}  
\usepackage{accents}     

\newtheoremstyle{mythmstyle}
  {\topsep}
  {\topsep}
  {\itshape}
  {}
  {\bfseries \sffamily}
  {.}
  {.5em}
  {}

\newtheoremstyle{mydefstyle}
  {\topsep}
  {\topsep}
  {\normalfont}
  {}
  {\bfseries \sffamily}
  {.}
  {.5em}
  {}

\theoremstyle{mythmstyle}
\newtheorem{thm}{Theorem}[section]      
\newtheorem{prop}[thm]{Proposition}  
\newtheorem{cor}[thm]{Corollary}      
\newtheorem{conjecture}[thm]{Conjecture}      
\newtheorem{lem}[thm]{Lemma}           
\theoremstyle{mydefstyle}
\newtheorem{defin}[thm]{Definition}
\newtheorem{ex}[thm]{Example}      
\newtheorem{rmk}[thm]{Remark}

\newcommand{\e}{\varepsilon}
\newcommand{\eps}{\varepsilon}  

\DeclareMathOperator{\Xe}{\mathit X_{\varepsilon}}
\DeclareMathOperator{\Xei}{\mathit X_{\varepsilon}^{\mathit{i}}}

\DeclareMathOperator{\Ee}{\mathit X_{\varepsilon,\mathit{e}}}

\DeclareMathOperator{\Ve}{\mathit X_{\varepsilon,\mathit{v}}}
\DeclareMathOperator{\V}{\mathit X_{\mathit{v}}}
\DeclareMathOperator{\Ye}{\mathit Y_{\varepsilon,\mathit{e}}}
\DeclareMathOperator{\Y}{\mathit Y_{\mathit{e}}}
\DeclareMathOperator{\dvol}{dvol}
\DeclareMathOperator{\Uv}{\mathit U_{\varepsilon,\mathit{v}}}
\DeclareMathOperator{\Xve}{\mathit X_{\varepsilon,\mathit{v},\mathit{e}}}
\DeclareMathOperator{\dom}{dom}
\DeclareMathOperator{\vol}{vol}

\DeclareMathOperator{\const}{const}

\newcommand{\normBC}{{\mathrm{norm}}} 
\newcommand{\tangBC}{{\mathrm{tan}}}  

\newcommand{\absBC}{{\mathrm{abs}}}  
\newcommand{\relBC}{{\mathrm{rel}}}  

\newcommand{\abs}[2][{}]{\lvert{#2}\rvert_{{#1}}}    

\ifthenelse{\isundefined \labelForExCoexAsWords}
{\newcommand{\exact}[1]{\bar{#1}}         
\newcommand{\coexact}[1]{\accentset={#1}}
\newcommand{\exactEV}[2]{\exact \lambda^{#2}_{#1}}
\newcommand{\exactEVabsBC}[2]{\exact \lambda^{#2,\absBC}_{#1}}
\newcommand{\coexactEV}[2]{\coexact \lambda^{#2}_{#1}}
\newcommand{\exactLapl}[2]{\exact \Delta^{#2}_{#1}}
\newcommand{\coexactLapl}[2]{\coexact \Delta^{#2}_{#1}}
}
{
\newcommand{\exact}[1]{#1^{\mathrm{ex}}} 
\newcommand{\coexact}[1]{#1^{\mathrm{co\text-ex}}} 
\newcommand{\exactEV}[2]{\lambda^{#2,\mathrm{ex}}_{#1}}
\newcommand{\coexactEV}[2]{\lambda^{#2,\mathrm{co-ex}}_{#1}}
\newcommand{\exactLapl}[2]{\Delta^{#2,\mathrm{ex}}_{#1}}
\newcommand{\coexactLapl}[2]{\Delta^{#2,\mathrm{co-ex}}_{#1}}
}

\newcommand{\wt}[1]{\widetilde{#1}}
\renewcommand{\phi}{\varphi}  

\newcommand{\Err}[1]{\mathcal O(#1)}

\newcommand{\specsymb} {\sigma} 
\newcommand{\spec}[2][{}]   {\specsymb_{\mathrm{#1}}(#2)}

\newcommand{\orient}[1]{\accentset{\curvearrowright}{#1}}

\newcommand{\R}{\mathbb{R}} 
\newcommand{\C}{\mathbb{C}} 
\newcommand{\N}{\mathbb{N}} 
\newcommand{\Sphere}{\mathbb{S}} 

\newcommand{\set}[2]{\{ #1 \, | \, #2 \} }      
\newcommand{\bigset}[2]{\bigl\{ #1 \, \bigl|\bigr. \, #2 \bigr\} }
\newcommand{\Bigset}[2]{\Bigl\{ #1 \, \Bigl|\Bigr. \, #2 \Bigr\} }
\newcommand{\BIGset}[2]{\Bigg\{ #1 \, \Bigg|\Bigg. \, #2 \Bigg\} }

\DeclareMathOperator*{\bigdcup}{\mathaccent\cdot{\bigcup}}
\DeclareMathOperator*{\dcup}   {\mathaccent\cdot\cup}

\DeclareMathOperator{\dd}    {d\!}  
\DeclareMathOperator{\id}    {id}  

\newenvironment{myitemize}{
\begin{itemize}
  \setlength{\itemsep}{1pt}
  \setlength{\parskip}{0pt}
  \setlength{\parsep}{0pt}
}{\end{itemize}}

\usepackage{ifthen}       
\usepackage{scrtime}      

\ifthenelse{\isundefined \draft }
{
  \newcommand{\look}[1]{}
  \newcommand{\lookO}[1]{}%
  \newcommand{\lookM}[1]{}%

  \automark[subsection]{section}  %
}
{
  \ifthenelse{\isundefined \Olaf}  
  {\usepackage[notref,notcite]{showkeys}}
  {\usepackage{/home/post/Aktuell/LaTeX/my-showkeys}}
  
  \newcommand{\markerO}{\fbox{\rule{0pt}{0.1ex}\textbf{Olaf}}}
  \newcommand{\markerM}{\fbox{\rule{0pt}{0.1ex}\textbf{Michela}}}
  \newcommand{\look}[1]{\textbf{*}
    \footnote{ #1 }}
  \newcommand{\lookO}[1]{\markerO\textbf{*}
    \footnote{\textbf{Olaf:} #1 }}
  \newcommand{\lookM}[1]{\markerM\textbf{*}
    \footnote{\textbf{Michela:} #1 }}

  \rehead{\parbox[t]{1.0\textwidth}%
    {\textsf{\texttt{\jobname.tex} \quad --- DRAFT --- \quad 
        Michela Egidi, Olaf Post\\
        \today, \thistime}}}
  \lohead{\parbox[t]{1.0\textwidth}%
    {\textsf{\texttt{\jobname.tex} \quad --- DRAFT ---  \quad 
        Michela Egidi, Olaf Post\\
        \today, \thistime}}}
}

\title{Asymptotic behaviour of the Hodge Laplacian spectrum on  graph-like
  manifolds
  \ifthenelse{\isundefined \draft}{}
  {\newline \upshape{--- DRAFT---}}}

\author{Michela Egidi, Olaf Post}

\dedication{\normalsize{Department of Mathematical Sciences,
    Durham University, England, UK\newline E-mail:
     \texttt{michela.egidi@durham.ac.uk}, 
     \texttt{olaf.post@durham.ac.uk}}}

\ifthenelse{\isundefined \draft}
{\date{\today}}  
{\date{\today, \thistime,  \emph{File:} \texttt{\jobname.tex}}} 

\begin{document} 

\maketitle 

\begin{abstract} 
  We consider a family of compact, oriented and connected
  $n$-dimensional manifolds $\Xe$ shrinking to a metric graph as
  $\e\to 0$ and describe the asymptotic behaviour of the eigenvalues
  of the Hodge Laplacian on $\Xe$.  We apply our results to produce
  manifolds with spectral gaps of arbitrarily large size in the
  spectrum of the Hodge Laplacian.
\end{abstract}

\tableofcontents
\section{Introduction}
\label{intro}
%

\subsection{Motivation}
A graph-like manifold is a family of compact, oriented and connected
$n$-dimensional Riemannian manifolds $\{X_{\e}\}_{\e>0}$ made of
building blocks according to the structure of a given metric graph,
i.e.\ a graph where each edge is associated a length.  The manifolds
$\Xe$ have the property that they shrink to the metric graph as $\e\to
0$.  A graph-like manifold is constructed from \emph{edge
  neibhourhoods} $\Ee=[0,\ell_e] \times \e \Y$ and \emph{vertex
  neighbourhoods} $\Ve$ according to the underlying graph.  The
shrinking parameter $\e$ is, roughly speaking, the radius of the
tubular neighbourhood, or in other words, the length scaling factor of
the transversal manifold $\Y$ at an edge $e$.  A precise definition is
given in Section~\ref{sec:g-like-mfds}.

Graph-like manifolds have been used in purely mathematical contexts as
well as in applications in Physics.  One prominent example in spectral
geometry is given by Colin de Verdi\`ere in~\cite{colin:86b}, where he
proved that the first non-zero eigenvalue of a compact manifold of
dimension $n \ge 3$ can have arbitrary large multiplicity.  In
Physics, graph-like manifolds, or more concrete, small neighbourhoods
of metric graphs embedded in $\R^n$ are used to model electronic or
optic nano-structures.  The natural question arising is if the
underlying metric graph is a good approximation for the graph-like
manifold.

The Laplacian $\Delta^0_{\Xe}$ on \emph{functions} on graph-like
manifolds has been analysed in detail, and the convergence of various
objects such as resolvents (in a suitable sense), spectrum etc.\ is
established in many contexts, see again~\cite{post:12,exner-post:13}
for more details and references.

\subsection{Aim of this article and main results}
The aim of this article is to consider the eigenvalues of the
Laplacian $\Delta^\bullet_{\Xe}$ acting on \emph{differential forms}
on the graph-like manifold and analyse their behaviour as $\e\to 0$.
To fix the notation in more detail, denote by $\Delta_{\Xe}^p$ the
Laplacian acting on $p$-forms on $\Xe$.  Any $p$-form can be
decomposed into its exact, co-exact and harmonic component
(see~\eqref{eq:hodge} for details), and the Laplacian leaves this
decomposition invariant.  We denote the $j$-th eigenvalue of the
Laplacian acting on exact resp.\ co-exact $p$-forms on $\Xe$ counted
with respect to multiplicity by $\exactEV j p(\Xe)$ resp.\ $\coexactEV
j p(\Xe)$ ($j=1,2,\dots$) and call them for short \emph{exact} and
\emph{co-exact eigenvalues}.  (Throughout this article, we will use
the labels $\exact \bullet$ and $\coexact \bullet$ for exact and
co-exact eigenvalues of Laplacians, respectively.)  By Hodge duality
and ``supersymmetry'' (the exterior derivative $d$ resp.\ its (formal)
adjoint $d^*$ is an isomorphism between co-exact and exact eigenforms
resp.\ vice versa, see the proof of Theorem~\ref{thm:exact}), we have
\begin{equation}
  \label{eq:trivial-ev-rel}
  \exactEV j p(\Xe)
  = \coexactEV j {n-p} (\Xe)
  \quad\text{and}\quad
  \exactEV j p (\Xe)
  = \coexactEV j {p-1}(\Xe),
  \qquad (j=1,2,\dots)
\end{equation}
so that it suffices to consider only the exact eigenvalues $\exactEV j
p (\Xe)$ for $1 \le p < n/2+1$ or the co-exact eigenvalues $\coexactEV
j p (\Xe)$ for $0 \le p < n/2$. 
In particular, if $n=2$, then the entire (non-zero) spectrum of the
differential form Laplacian is determined by its Laplacian on
functions.  This case can be considered as ``trivial'' in this
article, as the convergence for functions has already been established
in earlier works (see again~\cite{post:12,exner-post:13} and
references therein).

On the metric graph, we have also the notion of $p$-forms, but only
co-exact $0$-forms and exact $1$-forms are non-trivial (apart from the
harmonic forms determined by the topology of the graph).  We denote
the spectrum of the Laplacian on $0$-forms (functions) and (exact)
$1$-forms on the graph $X_0$ by $\lambda_j^0(X_0)$ and
$\lambda_j^1(X_0)$, respectively ($j=1,2,\dots$). Note that we have
$\lambda_j^1(X_0)=\lambda_j^0(X_0)$ so that we simply write
\begin{equation}
  \label{eq:trivial-ev-rel-graph}
  \lambda_j(X_0) 
  :=\lambda_j^1(X_0)
  =\lambda_j^0(X_0)
  \qquad\qquad (j=1,2,\dots)
\end{equation}
and speak of \emph{the} eigenvalues of the metric graph.  To be
consistent with the limit, we also set $\lambda_j(\Xe) := \exactEV j
1(\Xe)=\coexactEV j 0(\Xe)$.

Recall that $\Y$ denotes the transversal manifold at the edge $e$.
Our main result of this article is now the following:
\begin{thm}
  \label{thm:main}
  Let $\Xe$ be a graph-like Riemannian $n$-dimensional compact
  manifold and $X_0$ its underlying metric graph, then the following
  is true:
  \begin{subequations}
    \label{eq:ev-forms}
    \begin{enumerate}
    \item
      \label{thm.main.i}
      The $0$-form eigenvalues, or equivalently, the exact
      $1$-form eigenvalues of $\Xe$ converge to the eigenvalues of
      $X_0$, i.e.,
      \begin{equation}
        \label{eq:ev-0-forms}
        \lambda_j(\Xe) = \exactEV j 1(\Xe) 
        \underset{\e\to 0}{\longrightarrow}
        \lambda_j(X_0)
      \end{equation}
      for all $j=1,2,\dots$.
    \item
      \label{thm.main.ii}
      Assume that $n \ge 3$ and $2 \le p \le n-1$, and that all
      transversal manifolds $\Y$ have trivial $(p-1)$-th cohomology group
      (i.e., $H^{p-1}(Y_e)=0$ for all edges $e$), then we have
      \begin{equation}
        \label{eq:ev-1-forms}
        \exactEV 1 p(\Xe)
        =\coexactEV 1 {p-1}(\Xe)
        \underset{\e\to 0}{\longrightarrow}
        \infty
      \end{equation}
      for the first eigenvalue of exact $p$-forms on $\Xe$.
    \end{enumerate}
  \end{subequations}
\end{thm}
As a consequence of our eigenvalue ordering, all other eigenvalues
$\exactEV j p(\Xe)$ ($j=1,2,\dots$) for $2 \le p \le n-1$ diverge,
too.  We emphasise that the first part on the eigenvalue convergence
is a simple consequence of the convergence for the eigenvalues on
$0$-forms (functions) already established in earlier works (see
again~\cite{exner-post:05, post:12} and references therein).

We say that the graph-like manifold $X_\eps$ is \emph{transversally
  trivial} if all transversal manifolds $\Y$ are Moore spaces, i.e.,
they have trivial cohomology in the sense that $H^p(\Y)\ne 0$ only for
$p=0$ and $p=n-1$ (see Example~\ref{ex:gl-mfd} for a construction of
such manifolds).  

If we assume that $X_\eps$ is transversally trivial, we can summarise
our result as follows:
\begin{myitemize}
\item For functions, the convergence result has been established
  before. Duality gives convergence for $n$-forms.
\item For \emph{exact} $1$-forms, we also have convergence due to
  supersymmetry.  Duality gives also convergence for co-exact
  $(n-1)$-forms.
\item For \emph{co-exact} $1$-forms and hence exact $(n-1)$-forms,
  we have divergence (this case only appears if $n \ge 3$).
\item All other $p$-forms with $2 \le p \le n-2$ diverge (this case
  only appears if $n \ge 4$).
\end{myitemize}
If the cohomology of $\Y$ is non-trivial, then the above list is still
true, but divergence only happens for higher eigenvalues, i.e.,
$\exactEV j p(\Xe) =\coexactEV j {p-1}(\Xe) \longrightarrow \infty$ as
$\e\to 0$ for $j \ge N$, where $N$ can be computed using the
$(p-1)$-Betti numbers of the transversal manifolds $\Y$, see
Theorem~\ref{prop2} for details.  It remains an open question what
happens to the first $N-1$ (non-zero) eigenvalues (see
Rem.~\ref{rmk:eigenvalues}).

As a consequence of the above theorem, we obtain the following result
(for the notion of Hausdorff convergence, see
Section~\ref{sec:hausdorff}):
\begin{cor}
  \label{cor:main.conv}
  Assume that the graph-like manifold is transversally trivial (i.e.,
  all transversal manifolds $\Y$ have trivial cohomology for
  $p=1,\dots, n-2$).  Then the spectrum of the differential form
  Laplacian converges to the spectrum of the metric graph.  More
  precisely, for all $\lambda_0>0$ we have that
  $\sigma(\Delta^\bullet_ {\Xe}) \cap [0,\lambda_0]$ converges in
  Hausdorff distance to $\sigma(\Delta_{X_0}) \cap [0,\lambda_0]$.
\end{cor}

Furthermore, we asked ourselves about the relation between spectral
gaps in the spectrum of the Laplacian acting on 1-forms on $\Xe$ and
$X_0$, i.e., about intervals $(a,b)$ \emph{not} belonging to the
spectrum.  In particular, we have as immediate consequence of the
asymptotic description of the spectrum in Theorem~\ref{thm:main}
resp.\ Corollary~\ref{cor:main.conv} the following result on spectral
gaps (i.e.intervals \emph{disjoint} with the spectrum):
\begin{cor}
  \label{cor:main}
  Assume that the graph-like manifold is transversally trivial, and
  suppose that $(a_0,b_0)$ is a spectral gap for the metric graph
  $X_0$ then there exist $a_\e$, $b_\e$ with $a_\e \to a_0$ and $b_\e
  \to b_0$ such that $(a_\e, b_\e)$ is a spectral gap for the Hodge
  Laplacian on $\Xe$ on all degrees, i.e.,
  $\sigma(\Delta^\bullet_{\Xe}) \cap (a_\e,b_\e)=\emptyset$.
\end{cor}
In our applications below, $a_\eps=a=0$ (as $0$ is always an
eigenvalue of the (entire) Hodge Laplacian on all degrees), hence we
can choose $(0,b_\eps)$ as common spectral gap.

\subsection{Related works}

\paragraph{Bounds on first non-zero eigenvalues:}

There has been some work about the spectral gap at the bottom of the
spectrum (i.e.\ estimates of the first non-zero eigenvalue).  For
example, one can consider the quantity
\begin{equation}
  \label{eq:lambda1.est}
  \kappa(L,X,g) := \lambda_1(L) (\vol_n(X,g))^{m/n},
\end{equation}
where $L$ is an elliptic operator of order $m$ on the $n$-dimensional
compact Riemanian manifold $(X,g)$ (the powers assure that $\kappa$ is
scale-invariant, i.e., if $\wt g=\tau^2g$ for some constant $\tau>0$,
then $\wt L=\tau^{-m}L$ and $\vol_n(X,\wt g)=\tau^n(\vol_n(X,g)$,
hence $\kappa(\wt L,X,\wt g)=\kappa(L,X,g)$.  Berger~\cite{berger:73}
asked whether
 \begin{equation*}
   \sup_{\text{$g$ metric on $X$}} \kappa(\Delta_{(X,g)},X,g)
\end{equation*}
is finite on a given manifold $X$.  The answer is yes in dimension $2$
with constant 
\begin{equation}
  \label{eq:kappa.n=2}
  \kappa(\Delta_{(X,g)},X,g)
  \le 8 \pi(\gamma(X)+1)
\end{equation}
for a surface of genus $\gamma(X)$ (\cite{yang-yau:80}).  Our analysis
shows that the bound is optimal in the sense that for any $\delta>0$
there is a sequence of Riemannian surfaces $(X_i,g_i)$ of genus
$\gamma(X_i) \to \infty$ (graph-like manifolds based on Ramanujan
graphs) such that
\begin{equation}
  \label{eq:kappa.n=2.counterex}
  \kappa(\Delta_{(X_i,g_i)},X_i,g_i) 
  \eqsim \gamma(X_i)^{1-\delta},
\end{equation}
i.e., the bound $\gamma(X_i)$ in~\eqref{eq:kappa.n=2} is
asymptotically optimal (see Corollary~\ref{cor:div.fcts.n=2}).  This
result is very much in the spirit of~\cite{colbois-girouard:14}, where
Colbois and Giruoard construct graph-like manifolds $X_i$ based on
Ramanujan graphs.  They use as operator $L$ the Dirichlet-to-Neumann
operator on $\partial X_i$ with metric $h_i$, and they show that
$\kappa(\Lambda_{(\partial X_i,h_i)},\partial X_i,h_i)/\gamma(X_i)$ is
uniformly bounded from below, in particular,
$\kappa(\Lambda_{(\partial X_i,h_i)},\partial X_i,h_i) \to \infty$ as
$i \to \infty$.

For higher dimensions, the answer to the finiteness of
$\kappa(\Delta_{(X,g)},X,g)$ is no, as on a compact manifold $X$ of
dimension $3$ or higher, there are sequences of metrics $g_i$ such
that $\kappa(\Delta_{(X,g_i)},X,g_i) \to \infty$ as $i \to \infty$
(see e.g.\cite{colbois-dodziuk:94} and references therein).

For $L$ being the Laplacian on $p$-forms ($2 \le p \le n-2$) with $n
\ge 4$, there are also examples of metrics $g_i$ on a given manifold
such that $\kappa(L,X,g_i)$ tends to infinity
(see~\cite{gentile-pagliara:95}).  Actually, they proved the result
for \emph{exact} $p$-forms ($2 \le p \le n-1$), allowing also $n=3$.
We rediscover their result (Proposition~\ref{prp:ex.g-p}).
Unfortunately, it seems to be impossible to construct a sequence of
metrics such that
$\kappa^\bullet_i:=\kappa(\Delta_{(X,g_i)}^\bullet,X,g_i)$ tends to
infinity for the \emph{entire} Hodge Laplacian
$\Delta_{(X,g_i)}^\bullet$ acting on all degrees (including
functions), as on any manifold $X$ of dimension $n \ge 3$ there is a
sequence of metrics such that the (rescaled) exact $p$-form spectrum
diverges, while the function spectrum converges
(Corollary~\ref{cor:any.mfd}).
\begin{conjecture}
  We conjecture that on a compact manifold $X$ of dimension $n \ge 3$,
  we have a uniform bound on the entire (non-zero) Hodge Laplace
  spectrum
  \begin{equation}
    \label{eq:kappa.hodge}
    \sup_{\text{$g$ metric on $X$}} 
    \kappa(\Delta_{(X,g)}^\bullet,X,g) < \infty.
  \end{equation}
\end{conjecture}
This conjecture is consistent with the case $n=2$, as the Hodge
Laplacian on a surface is entirely determined by the function
spectrum, and hence~\eqref{eq:kappa.hodge} holds as well.

Colbois and Maerten~\cite{colbois-maerten:10} have shown that on any
compact manifold $X$ of dimension $n \ge 2$ there is a sequence of
metrics $g_i$ such that $\kappa(\nabla^*\nabla^p_{(X,g_i)},X,g_i) \to
\infty$, where $\nabla^*\nabla^p_{(X,g_i)}$ denotes the rough
(Bochner) Laplacian on $p$-forms, $1\le p \le n-1$.

\paragraph{Other related works:}
There is another research line for ``collapsing'' manifolds, where one
considers families of manifolds with singular limit, but under some
curvature bounds.  Such a limit usually induces extra structure.  We
refer to~\cite{jammes:05} and \cite{lott:02} or~\cite{lott:14} for an
overview.  Let us stress that our graph-like manifolds $X_\eps$ always
have curvature tending to $\pm \infty$ as $\eps \to 0$.

Moreover, Jammes showed in~\cite{jammes:11} (see also the references
therein) that on any compact manifold of dimension $n \ge 6$, one can
prescribe the volume and any finite part of the spectrum of the Hodge
Laplacian acting on $p$-forms if $1 \le p < n/2$ and $n \ge 6$ (in the
spirit of Colin de Verdi\`ere~\cite{colin:86b,colin:87}, who treated
the function case $p=0$).  There are related results on constructing
metrics such that the Hodge Laplacian has certain spectral properties
also in~\cite{guerini:04,guerini-savo:04}.  In all these ``spectral
engineering'' papers, graph-like manifolds play a prominent role.
Chanillo and Treves showed in~\cite{chanillo-treves:97} a lower bound
on the entire Hodge Laplace spectrum in terms of certain
``admissable'' coverings of the manifold.

\subsection{Organisation of the paper}
The paper is organised as follows. In Section~\ref{sec:graph} we
briefly describe discrete and metric graphs and their associated
Laplacians on functions and on $1$-forms.  In
Section~\ref{sec:hodge-lapl} we review some facts on Laplacians on
differential forms as well as some useful facts about their
eigenvalues.  Section~\ref{sec:g-like-mfds} is dedicated to the
description of graph-like manifolds and their harmonic
forms. Section~\ref{sec:proof} contains the proof of our main result
Theorem~\ref{thm:main}, namely the convergence of exact $1$-forms and
the divergence result for higher degree forms.  Finally in
Section~\ref{sec:examples} we apply our results to establish the
existence of spectral gaps for families of metric graphs and their
graph-like manifolds.  Moreover, we construct example of (families of)
manifolds with (upper or lower) bounds on the first eigenvalue of the
Laplacian acting on functions or forms.

\subsection{Acknowlegements}
OP would like to thank Bruno Colbois for a helpful discussion and for
pointing his attention to~\cite[Lemma~2.3]{mcgowan:93} (the
``McGowan'' lemma, see Proposition~\ref{prop:mcgowan} in our paper).

\section{Discrete and metric graphs and their Laplacians}
\label{sec:graph}
%

The following material sets the notation; for more references and
details we refer e.g.\ to~\cite{post:12,berkolaiko-kuchment:13} and
the references therein.  The consideration of functions and forms on
discrete and metric graphs can also be found in~\cite{post:09c}, see
also~\cite{gaveau-okada:91}.

\subsection{Discrete graphs and their Laplacians}
\label{subsec:discrete_graph}

Let $G=(V,E,\partial)$ be a finite discrete graph, i.e., $V=V(G)$ and $E=E(G)$
are finite sets (\emph{vertices} and \emph{edges} respectively) and
$\partial \colon E \rightarrow V\times V$ is such that $e \mapsto
(\partial_-e,\partial_+e)$ associates to an edge its initial and
terminal vertex fixing an orientation for the graph, crucial when
working with $1$-forms.
\emph{We assume (without stating each time) that all discrete graphs
  are connected}.

For each vertex $v\in V$ we denote with
\begin{equation*}
  E_v^{\pm}=\set{e\in E} {\partial_{\pm}e=v}
\end{equation*}
the set of \emph{incoming} and \emph{outgoing edges} at a vertex $v$
and with
\begin{equation*}
  E_v=E_v^-\dcup E_v^+
\end{equation*}
(disjoint union) the set of vertices emanating from $v$.  The
\emph{degree of a vertex} is the number of emanating edges, i.e.,
\begin{equation*}
  \deg v:=|E_v|.
\end{equation*}
Note that we allow loops, i.e., $\partial_-e=\partial_+ e=v$, and each
loop is counted twice in $\deg v$ (as we have taken the disjoint union
in $E_v^- \dcup E_v^+$). We also allow multiple edges, i.e., edges
with the same starting and ending point.

Assume that $G=(V,E,\partial)$ is a discrete graph and $\ell \colon E
\longrightarrow (0,\infty)$ a map associating to each edge a number
$\ell_e >0$ (its ``length'', as we will interpret it below).  Given a
function $F \colon V \longrightarrow \C$ on the vertex space of $G$
(or a vector $F \in V^\C$, if you prefer), the \emph{discrete
  Laplacian} (on functions) $\Delta_G = \Delta_G^0$ is defined as
\begin{equation*}
  (\Delta_G F)(v)=
  -\frac{1}{\deg v}\sum_{e\in E_v}
  \frac{1}{\ell_e}\big(F(v_e)-F(v)\big),
\end{equation*}
where $v_e$ is the vertex on the opposite site of $v$ on $e \in E_v$.
We note that $\Delta_G$ can also be defined as $\Delta_G=d^*_G d_G$
where
\begin{equation*}
  d_G \colon \ell_2(V,\deg) \longrightarrow \ell_2(E,\ell^{-1}),\qquad 
  (d_GF)_e=F(\partial_+e)-F(\partial_-e),
\end{equation*}
and where $\ell_2(V,\deg)=\C^V$ resp.\ $\ell_2(E,\ell^{-1})=\C^E$
carry the norms given by
\begin{equation*}
  \|F\|^2_{\ell_2(V,\deg)}
  =\sum_{v\in V}|F(v)|^2\deg v
  \qquad\text{and}\qquad
  \|\eta\|^2_{\ell_2(E,\ell^{-1})}
  =\sum_{e\in E}|\eta|^2\frac{1}{\ell_e}
\end{equation*}
and where $d^*_G$ is its adjoint operator with respect to the
corresponding inner products.  We can equally define a Laplacian on
$1$-forms by $\Delta_G^1 := d_G d_G^*$, acting on
$\ell_2(E,\ell^{-1})$.

\subsection{Metric graphs and their Laplacians}
\label{subsec:metric_graph}

Let $G=(V,E,\partial)$ be a discrete graph and $\ell \colon
E\rightarrow(0,\infty)$ a function associating to each edge $e \in E$
a number $\ell_e>0$ which we will interpret as \emph{length} as
follows. We define a \emph{metric graph} associated with the discrete
graph $G$ as the quotient
\begin{equation*}
  X_0:=\bigdcup_{e \in E} I_e/{\thicksim},
\end{equation*} 
where $I_e:=[0,\ell_e]$ and $\thicksim$ is the relation identifying
the end points of the intervals $I_e$ according to the graph: namely,
$x \sim y$ if and only if $\psi(x)=\psi(y)$ where $\psi \colon
\bigdcup_{e \in E} I_e \to V$, $0 \in I_e \mapsto \partial_-e$,
$\ell_e \in I_e \mapsto \partial_+e$ and $\psi(x)=x$ for $x \in
\bigdcup_{e \in E} (0,\ell_e)$.

On a metric graph, we have a natural measure (the Lebesgue measure on
each interval) allowing us to define a natural Hilbert space of
functions and $1$-forms by
\begin{equation*}
  L^2(X_0)=\bigoplus_{e\in E}L^2(I_e)
  \quad\text{and}\quad
  L^2(\Lambda^1(X_0))=\bigoplus_{e\in E}L^2(\Lambda^1(I_e))
\end{equation*}
with norms given by
\begin{equation*}
  \|f\|^2_{L^2(X_0)}:=\sum_{e\in E}\|f_e\|^2_{L^2(I_e)}
  \quad\text{and}\quad
  \|\alpha\|^2_{L^2(\Lambda^1(X_0))}
  :=\sum_{e\in E}\|\alpha_e\|^2_{L^2(\Lambda^1(I_e))}
\end{equation*}
for functions $f \colon X_0\longrightarrow \C$, $f=(f_e)_{e\in E}$ and
$1$-forms $\alpha=(\alpha_e)_{e\in E} =(g_e\, \dd s_e)_{e\in E}$ on $X_0$,
respectively. Note that functions on $I_e$ can obiouvsly be identified
with $1$-forms via $g_e \mapsto g_e \dd s_e$; the difference of forms and
functions appears only in the domain of the corresponding differential
operators below.

We define the \emph{exterior derivative} $d=d_{X_0}$ on $X_0$ as the
operator
\begin{equation*}
  d \colon \dom d \longrightarrow L^2(\Lambda^1(X_0)),
  \quad d(f_e)_{e\in E}=(f'_e \dd s_e)_{e\in E}
\end{equation*}  
with domain
\begin{equation*}
  \dom d = H^1(X_0) \cap C(X_0)
\end{equation*}
where $H^1(X_0)=\set{f\in L^2(X_0)}{f'=(f_e')_{e\in E} \in L^2(X_0)}$
and where $C(X_0)$ denotes the space of continuous functions on $X_0$.

It is not difficult to see that $d=d_{X_0}$ is a \emph{closed}
operator with adjoint given by
\begin{equation*}
  d^* (g_e \dd s_e)_{e \in E}  = -(g'_e )_{e\in E}
\end{equation*}
with domain
\begin{equation*}
  \dom d^* =
  \Bigset{\alpha \in H^1(\Lambda^1(X_0))}
         {\sum_{e\in E} \orient{\alpha}_e(v) = 0}
\end{equation*}
with $H^1(\Lambda^1(X_0))=\set{\alpha =(g_e \dd s_e)_{e \in E} \in
  L^2(\Lambda^1(X_0))}{g_e' \in L^2(I_e)}$ where $\orient{\alpha}$ is the oriented evaluation of
$\alpha$, i.e.,
\begin{equation*}
  \orient{\alpha}_e(v) =
  \begin{cases}
    -g_e(0), & v=\partial_-e\\
    g_e(\ell_e),& v=\partial_+e.
  \end{cases}
\end{equation*}

The Laplacians acting on functions and $1$-forms defined on $X_0$ are
the operators
\begin{align*}
  \Delta^0_{X_0}&=d^*d
  \quad\text{where}\quad
  \dom\Delta^0_{X_0}
  =\set{\alpha \in \dom d} {d \alpha \in \dom d^*} 
  \qquad\text{and}\\
  \Delta^1_{X_0}&=d d^*
  \quad\text{where}\quad
  \dom \Delta^1_{X_0}
  =\set{\alpha \in \dom d^*} {d^*\alpha \in \dom d}.
\end{align*}
Writing the vertex conditions for the Laplacian on functions
explicitly, we obtain the conditions
\begin{equation}
  \label{eq:kirchhoff}
  f_e(v):=
  \begin{cases}
    f_e(0),&v=\partial_-e\\
    f_e(\ell_e),&v=\partial_-e
  \end{cases}
  \quad  \text{is independent of $e \in E_v$ and} \quad
  \sum_{e \in E_v} \orient f'_e(v)=0,
\end{equation}
called \emph{standard} or \emph{Kirchhoff} vertex conditions.  The
first condition can be rephrased as \emph{continuity} of $f$ on the
metric graph, while the second is a \emph{flux conservation}
considering the derivative $df=(f_e')_e$ as vector field.
\begin{rmk}
  \label{rem:supersymmetry}
  We remark that $\Delta_{X_0}^0=d_{X_0}^*d_{X_0}$ and
  $\Delta_{X_0}^1=d_{X_0}d_{X_0}^*$ are both non-negative operators
  and fulfil a ``supersymmetry'' condition in the sense
  of~\cite[Sec.~1.2]{post:09c}.  As a consequence, their non-zero
  spectrum including multiplicity are the same
  (\cite[Prop.~1.2]{post:09c}, see also the proof of
  Theorem~\ref{thm:exact}).  This remark applies also to the discrete
  graph Laplacians $\Delta_G^0=d_G^*d_G$ and $\Delta_G^1=d_G d_G^*$ of
  Section~\ref{subsec:metric_graph}, as well as to the co-exact and
  exact Laplacian $\coexactLapl \Xe {p-1}=d_{\Xe}^*d_{\Xe}$ and
  $\exactLapl \Xe p=d_{\Xe}d_{\Xe}^*$, respectively, explaining the
  second relation in~\eqref{eq:trivial-ev-rel}.
\end{rmk}

Finally, we remind the reader that whenever the graph is equilateral,
i.e., $\ell_e=\ell_0$ for all $e\in E$, the spectra of the discrete
Laplacian and the metric Laplacian on $0$-forms are related in the
following sense. Let $\Sigma := \{ (j\pi/\ell_0)^2 \mid j=1,2,\dots\}$
be the Dirichlet spectrum of the interval $[0,\ell_0]$, then
\begin{equation}
  \label{spectra_rel}
  \lambda \in \sigma(\Delta^0_{X_0})
  \qquad\text{if and only if}\qquad
  \phi(\lambda):= 1-\cos(\ell_0\sqrt{\lambda}) \in \sigma(\Delta_G)
\end{equation}
for all $\lambda \notin \Sigma$ (see e.g.~\cite{nicaise:85,cattaneo:97} or
\cite[Sec.~2.4.1]{post:12}), and we have the obvious relation also
for $1$-forms due to Remark~\ref{rem:supersymmetry}.  There is also a
relation at the bottom of the spectrum of $\Delta_G$ and
$\Delta_{X_0}$ for \emph{general} (not necessarily equilateral) metric
graphs for which we refer to \cite[Sec.~6.1]{post:09c} or
\cite[Sec.~2.4.2]{post:12} for more details.

\subsection{Discrete and metric Ramanujan graphs}
\label{subsec:ramanujan}

A discrete graph $G$ is $k$-regular, if all its vertices have degree
$k$.  For ease of notation we assume here that the graph
$G=(V,E,\partial)$ is \emph{simple}, and we write $v \sim w$ for
adjacent vertices.
\begin{defin}
  Let $G$ be a $k$-regular discrete graph with $n$ vertices and let
  $\Delta_G$ be its (normalised) discrete Laplacian.  The graph $G$ is
  said to be \emph{Ramanujan} if
  \begin{equation*}
    \max\bigset{|1-\mu|}{\mu \in \sigma(\Delta_G)}
    \leq \frac{2\sqrt{k-1}}{k}.
  \end{equation*}
\end{defin}
Note that many authors use the eigenvalues of the \emph{adjacency
  matrix} $A_G$ as the spectrum of a graph.  The adjacency matrix is
given by $(A_G)_{v,w}=1$ if $v \sim w$ and $(A_G)_{v,w}=0$. As $v \sim
w$ is equivalent with $w \sim v$, the adjacency matrix if symmetric.
For a $k$-regular graph, we have the relation
\begin{equation}
  \label{eq:disc.lap.adj}
  A_G = k(\id - \Delta_G),
  \qquad\text{or,}\qquad
  \Delta_G = \id - \frac 1 k A_G
\end{equation}
with the discrete graph Laplacian (with ``length'' function
$\ell_e=1$, see Section~\ref{subsec:discrete_graph}).  Note that
$\spec{A_G} \subset [-k,k]$ and $\spec{\Delta_G}\subset [0,2]$, and
that $k$, resp.\ $0$, is always an eigenvalue of $A_G$, resp.\ $\Delta_G$,
while $-k$, resp.\ $2$, is an eigenvalue of $A_G$, resp.\ $\Delta_G$, if
and only if the graph is bipartite (recall that we assume that $G$ is
a \emph{finite} graph).

We define the \emph{(maximal) spectral gap length} of a discrete graph
by
\begin{equation}
  \label{eq:graph.gap}
  \gamma(G) 
  := \min \bigset{ \mu, 2-\mu}{\mu \in \spec {\Delta_G} \setminus\{0,2\}}
  = 1-\frac 1k\max \bigset{ |\alpha| } {\alpha \in \spec {A_G}, |\alpha|<k}
\end{equation}
i.e., $\gamma(G)$ is the distance from the non-trivial spectrum $0$
(resp.\ $0$ and $2$ in the bipartite case) of the Laplacian $\Delta_G$
from $\{0,2\}$ resp.\ $\{-k,k\}$.  Hence, a graph is a Ramanujan graph
if its spectral gap length has size at least
\begin{equation*}
  \gamma(G)   \geq 1-\frac{2\sqrt{k-1}}{k}.
\end{equation*}
It has been shown that the lower bound is optimal, i.e., for any
$k$-regular graph (or even for any graph with maximal degree $k$) with
diameter large enough, the spectral gap length is smaller than
$1-2\sqrt{k-1}/k+\eta$, where $1/\eta$ is of the same order as the
diameter (see~\cite[Thm.~1]{nilli:91} and references therein).

The existence of \emph{infinite} families $\{G^i\}_{i \in \N}$ of
$k$-regular graphs has been shown whenever $k$ is a prime or a power
of a prime (see e.g.~\cite{lps:88, margulis:88,
  morgenstern:94}). Recently, the existence of infinite families of
regular bipartite Ramanujan graphs of every degree $k> 2$ has been
proved in~\cite{mss:pre14} by showing that any bipartite Ramanujan
graph has a $2$-lift which is again Ramanujan, bipartite and has twice
as many vertices.

Let $\{G^i\}_{i \in \N}$ be a family of Ramanujan graphs such that
\begin{equation}
  \label{eq:vx.ram}
  \nu_i := |V(G^i)| \to \infty
\end{equation}
and consider the associated family of equilateral metric graphs
$\{X_0^i\}_{i \in \N}$ of length $\ell_0$.  By~\eqref{spectra_rel},
the metric graph Laplacians $\Delta_{X_0^i}$ all have a spectral gap
\begin{equation}
  \label{eq:ramanujan-gap}
  (a_0,b_0) = \Bigl(0, \frac h {\ell_0^2} \Bigr)
  \qquad\text{with}\qquad
  h=h_k :=  \arccos^2 
  \Bigl(1-\frac{2\sqrt{k-1}} k  \Bigr)>0
\end{equation}
at the bottom of the spectrum.

\section{The Hodge Laplacian and their eigenvalues}
\label{sec:hodge-lapl}
%

In this section, we collect some general facts on differential forms,
the Hodge Laplacian and its spectrum.

\subsection{Differential forms and the Hodge Laplacian}
\label{sec:diff-forms}

Let $(M,g)$ be a compact, oriented and connected $n$-dimensional
Riemannian manifold.  Its Riemannian metric $g$ induces the
$L^2$-space of $p$-forms
\begin{equation*}
  L^2(\Lambda^p(M,g))
  =\Bigset{\omega \colon M \longrightarrow\C}
          { \|\omega\|^2_{L^2(\Lambda^p(M,g)} = \int_M|\omega|_g^2\dvol_g M<\infty}
\end{equation*}
where
\begin{equation*}
  \|\omega\|^2_{L^2(\Lambda^p(M,g))}
  =\langle \omega,\omega \rangle_{L^2(\Lambda^p(M,g))}
  := \int_M|\omega|_g^2 \dvol_g M
  =\int_M \omega\wedge\ast \omega
\end{equation*}
and where $\ast$ denotes the Hodge star operator (depending on $g$).
The Laplacian on $p$-forms on $M$ is formally defined as
$\Delta_{(M,g)}^p=\Delta^p=d\delta + d \delta$ where $d$ is the
classical exterior derivative and $\delta=(-1)^{np+n+1}{\ast} d \ast$
its formal adjoint with respect to the inner product induced by $g$.
If $M$ has no boundary, then $\delta$ is the $L^2$-adjoint of $d$ and
$\Delta^p$ is a non-negative self-adjoint operator with discrete
spectrum denoted by $\lambda_j^p(M,g)$ (repeated according to
multiplicity).

We allow that $M$ has a boundary $\partial M$, itself a smooth
manifold of dimension $n-1$.  As in the function case, it is possible
to impose boundary conditions for functions in the domain of the Hodge
Laplacian.  To do so, we first decompose a $p$-form $\omega$ in its
tangential and normal components on $\partial M$, i.e.,
$\omega=\omega_\tangBC+\omega_\normBC$ where $\omega_\tangBC$ can be
considered as a form on $\partial M$ while $\omega_\normBC=\dd r
\wedge \omega^\bot$ with $\omega^\bot$ being a form on $\partial M$
and $r$ being the distance from $\partial M$.

The Hodge Laplacian with \emph{absolute boundary conditions} is given
by those forms $\omega$ such that
\begin{equation*}
  \omega_\normBC=0
  \quad\text{and}\quad 
  (d\omega)_\normBC=0
\end{equation*}
while \emph{relative boundary conditions} require
\begin{equation*}
  \omega_\tangBC=0
  \quad\text{and}\quad 
  (\delta \omega)_\tangBC=0.
\end{equation*} 
These boundary conditions give rise to two unbounded and self-adjoint
operators $\Delta^\absBC$ and $\Delta^\relBC$ with discrete spectrum,
the \emph{Hodge Laplacians} with \emph{absolute} and \emph{relative
  boundary conditions}, respectively (see e.g.~\cite{chavel:84}
or~\cite{mcgowan:93}).  Recall that for functions, the absolute
correspond to Neumann while the relative correspond to Dirichlet
boundary conditions.

Furthermore, since the Hodge star operator exchanges absolute and
relative boundary conditions, there is a correspondence between the
spectrum of $\Delta^\absBC$ and the spectrum of $\Delta^\relBC$, which
allows us to study just one of them to cover both cases.  \emph{In the
  sequel, we will only consider \emph{absolute} boundary conditions if
  the manifold has a boundary, and hence we will mostly suppress the
  label $(\cdot)^\absBC$ for ease of notation.}

In an $L^2$-framework, we consider $d$ and $\delta_0$ as unbounded
operators, defined as the closures $\overline d$ and $\overline
\delta_0$ of $d$ and $\delta_0$ on
\begin{align*}
  \dom d 
  &= \bigset{\omega \in C^\infty(\Lambda^p(M,g))}
         {d \omega \in L^2(\Lambda^{p+1}(M,g))}\\
  \dom \delta_0 
  &= \bigset{\omega \in C^\infty(\Lambda^p(M,g))}
         {\delta \omega \in L^2(\Lambda^{p+1}(M,g)),\;
             \omega_\normBC=0},
\end{align*}
respectively.  The Hodge Laplacian with absolute boundary condition is
then given by
\begin{equation*}
  \Delta
  =\Delta^\absBC
  =\overline d \,\overline \delta_0 + \overline \delta_0 \overline d
\end{equation*}
For this operator, Hodge Theory is still valid.  In particular, the
de~Rham theorem holds (see~\cite[Sec.~2.1]{mcgowan:93} and references
therein): the space of harmonic $p$-forms (with absolute boundary
conditions if the boundary is non-empty) $\mathcal{H}^p(M,g)$, is
isomorphic to the $p$-th de Rham cohomology, $H^p(M)$, and any
$p$-form $\omega \in L^2(\Lambda^p(M,g))$ can be orthogonally
decomposed into an \emph{exact} ($d\exact \omega$), \emph{co-exact}
($\delta \coexact \omega$) and \emph{harmonic} ($\omega_0$) component,
i.e.,
\begin{equation}
  \label{eq:hodge}
  \omega =d\exact \omega  + \delta \coexact \omega  + \omega_0
\end{equation}
where $\exact \omega \in \dom \overline d$ is a $(p-1)$-form,
$\coexact \omega \in \dom \overline \delta_0$ is a $(p+1)$-form and
$\omega_0$ is a harmonic $p$-form.  Moreover, the Hodge Laplacian
leaves these spaces invariant and maps $p$-forms into $p$-forms.  In
particular, we can consider the eigenvalues of the Hodge Laplacian
acting on exact and co-exact $p$-forms as $\overline d \, \overline
\delta_0$ and $\overline \delta_0 \overline d$, called here
\emph{exact} and \emph{co-exact (absolute) $p$-form eigenvalues},
denoted by
\begin{equation*}
  \exactEV j p (M,g) 
  \quad\text{and}\quad
  \coexactEV j p (M,g),
\end{equation*}
respectively.  Let $\exact E^p(\lambda)=\ker(d\, \delta_0-\lambda)$
and $\coexact E^p(\lambda)=\ker(\delta_0 d -\lambda)$ denote the
eigenspaces of exact and co-exact $p$-forms with eigenvalue $\lambda$
(as the eigenforms are smooth by elliptic regularity, we can omit the
closures).  Since $d$ is an isomorphism between $\coexact
E^{p-1}(\lambda)$ and $\exact E^p(\lambda)$, we have the second
equality of~\eqref{eq:trivial-ev-rel}.  For the first equality, we use
the Hodge star operator (it interchanges absolute and relative
boundary conditions, but in~\eqref{eq:trivial-ev-rel} we only consider
boundaryless manifolds).

\subsection{An estimate from below for exact eigenvalues}
\label{subsec:estimate}

We now introduce a simplified but useful version of an estimate from
below on the first eigenvalue of the exact $p$-form Laplacian on a
manifold by McGowan (\cite[Lemma~2.3]{mcgowan:93}) also used in
Gentile and Pagliara in~\cite[Lemma~1]{gentile-pagliara:95}.

Let $(M,g)$ be a $n$-dimensional compact Riemannian manifold without
boundary and let $\{U_i\}_{i=1}^m$ be an open cover of $M$ such that
$U_{ij}=U_i\cap U_j$ have a smooth boundary.  
Moreover, we  denote by
\begin{equation*}
  I_i := \set{j \in \{1,\dots, i-1,i+1,\dots,m\}} 
  {U_i\cap U_j\neq \emptyset}.
\end{equation*}
the index set of \emph{neighbours of $U_i$}.  We say that the cover
$\{U_i\}_i$ has \emph{no intersection of degree $r$} iff $U_{i_1}\cap
\dots \cap U_{i_r}=\emptyset$ for any $r$-tuple $(i_1,\dots, i_r)$
with $1 \le i_1<i_2<\dots<i_r \le m$.  We choose a fixed partition of
unity $\{\rho_j\}_{j=1}^m$ subordinate to the open cover and we set
$\|d \rho\|_\infty := \max_i\,\sup_{x\in U_i}|d \rho_i(x)|_g$.

Furthermore, we denote by $\exactEVabsBC 1 p(U)$ the first positive
eigenvalue on exact $p$-forms on $U$ satisfying absolute boundary
conditions on $\partial U$.  Finally, denote by $H^p(U_{ij})$ the
$p$-th cohomology group of $U_{ij}$.

\begin{prop}
  \label{prop:gentile-pagliara:95}
  Let $M$ and $\{U_i\}_{i=1}^m$ be as above. Assume that the open
  cover has no intersection of degree higher than $2$ and
  $H^{p-1}(U_{ij})=0$ for all $i,j$. Then, the first positive
  eigenvalue of the Laplacian acting on exact $p$-forms on $M$
  satisfies
  \begin{equation*}
    \exactEV 1 p (M)
    \geq \dfrac{2^{-3}}
     {\displaystyle \sum_{i=1}^m   
       \Bigg(
         \dfrac 1 {\exactEVabsBC 1 p(U_i)}
         +\sum_{j \in I_i}
         \bigg(\dfrac{c_{n,p}\|d \rho\|^2_\infty}
                   {\exactEVabsBC 1 {p-1}(U_{ij})}+1
         \bigg)
         \bigg(\dfrac 1 {\exactEVabsBC 1 p(U_i)}
             +\dfrac 1 {\exactEVabsBC 1 p(U_j)}
        \bigg)
      \Bigg)
    }
  \end{equation*} 
  where $c_{n,p}$ is a combinatorial constant depending only on $p$
  and $n$.
\end{prop}

We remark that these assumptions impose a topological restriction on
the manifold as such an open cover does not necessarily exist.
Actually, the following general version holds for \emph{higher} exact
eigenvalues:
\begin{prop}
\label{prop:mcgowan}
Let $M$ and $\{U_i\}_i$ be as above and assume that the open cover has
no intersection of degree higher than $2$. We set $N_1=\sum_{i,j}\dim
H^{p-1}(U_{i,j})$ and $N=N_1+1$. Then, the $N$-th eigenvalue of the
Laplacian on exact $p$-forms on $M$ satisfies
  \begin{equation*}
    \exactEV N p (M)
    \geq \dfrac 1
     {\displaystyle \sum_{i=1}^m   
       \Bigg(
         \dfrac 1 {\exactEVabsBC 1 p(U_i)}
         +\sum_{j \in I_i}
         \bigg(\dfrac{\|d \rho\|^p_\infty}
                   {\exactEVabsBC 1 {p-1}(U_{ij})}+1
         \bigg)
         \bigg(\dfrac 1 {\exactEVabsBC 1 p(U_i)}
             +\dfrac 1 {\exactEVabsBC 1 p(U_j)}
        \bigg)
      \Bigg)
    }
  \end{equation*} 
\end{prop}

The proof of this proposition uses the same argument of the proof of
McGowan's lemma. The generalisation to $p$-forms is trivial since we
have particular assumptions on the cover, i.e., no intersections of
degree higher than $2$ (see the remark after the proof of his
``technical lemma''~\cite[Lemma~2.3]{mcgowan:93}).

\subsection{A variational characterisation of exact eigenvalues}
\label{subsec:characterization}

We will make use of the following characterisation of eigenvalues of
the Hodge Laplacian acting on exact $p$-forms by
Dodziuk~\cite[Prop.~3.1]{dodziuk:82} whose proof can be found
in~\cite[Prop.~2.1]{mcgowan:93}.  Its advantage is that it does not make
use of the adjoint $\delta$ of the exterior derivative, and hence the
metric $g$ does not enter in a complicated way.
\begin{prop}
  \label{prop:abs_exact}
  Let $M$ be a compact Riemannian manifold, then the spectrum of the
  Laplacian $0<\exactEV 1 p \le \exactEV 2 p \le \ldots$
  on exact $p$-forms on $M$ satisfying absolute boundary conditions
  can be computed by
  \begin{equation*}
    \exactEV j p (M)
    =\inf_{V_j} \sup
     \Bigset{ \frac {\langle \eta,\eta \rangle_{L^2(\Lambda^p(M))}}
                  {\langle \theta,\theta \rangle_{L^2(\Lambda^{p-1}(M))}}}
             {\eta\in V_j\setminus\{ 0\} \text{ such that }
                  \eta=d\theta},
  \end{equation*}
  where $V_j$ ranges over all $j$-dimensional subspaces of smooth
  exact $p$-forms and $\theta$ is a smooth $(p-1)$-form. 
\end{prop}

The advantage of this characterisation is that the metric only enters
via the $L^2$-norm, and no derivatives of the metric or its
coefficients are needed.

As a consequence we have (see~\cite[Prop.~3.3]{dodziuk:82}
or~\cite[Lem.~2.2]{mcgowan:93}):
\begin{prop}
  \label{prop:cont.dep}
  Assume that $g$ and $\wt g$ are two Riemannian metrics on $M$ such
  that $c_-^2g \le \wt g \le c_+^2 g$ for some constants $0<c_-\le c_+ <
  \infty$, i.e.,
  \begin{equation*}
    c_-^2g_x(\xi,\xi) \le \wt g_x(\xi,\xi) \le c_+^2 g_x(\xi,\xi)
    \quad\text{for all $\xi \in T_x^*M$ and $x \in M$},
  \end{equation*}
  then the eigenvalues of exact $p$-forms with absolute boundary
  conditions fulfil
  \begin{equation*}
    \frac 1{c_-^2} \Bigl(\frac {c_-}{c_+}\Bigr)^{n+2p} \exactEV j p(M,g)
    \le \exactEV j p(M, \wt g)
    \le \frac 1{c_+^2} \Bigl(\frac {c_+}{c_-}\Bigr)^{n+2p} \exactEV j p(M,g)
  \end{equation*}
  for all $j=1,2,\dots$
\end{prop}
As a consequence, the eigenvalues $\exactEV j p(M,g)$ depend
continuously on $g$ in the sup-norm. defined e.g.\
in~\cite[Sec.~5.2]{post:12}.  In particular, this proposition allows
us to consider also perturbation of graph-like manifolds as defined in
the next section.  For a discussion of possible cases we refer
to~\cite[Sec.~5.2--5.6]{post:12}).  As an example, we could consider
tubular neighbourhoods of graphs \emph{embedded} in $\R^n$.

\subsection{Scaling behaviour}
\label{sec:scaling}

We say that a Riemannian manifold $M_{\e}$ with a metric $g_{\e}$ is
\emph{$\e$-homothetic}, if $(M_{\e},g_{\e})$ is conformally equivalent
with a Riemannian manifold $(M,g)$ with (constant) conformal factor
$\e^2$, i.e., $g_\e=\e^2g$.  For short, we write $M_{\e}=\e M$.  It is
often convenient to think of the Riemannian manifold $M_{\e}$ as the
($\e$-independent) manifold $M$ with metric $g_{\e}=\e^2 g$.

Obviously, the scaling with a constant factor leads to the following
simple result for $p$-forms on a Riemannian manifold.

\begin{lem}
  \label{lem1}
  Let $\omega$ be a $p$-form on a $n$-dimensional Riemannian manifold
  $M$ with metric $g$, and let $\e M$ be the Riemannian manifold
  $(M, \e^2 g)$, then we have
  \begin{subequations}
    \begin{align}
      \label{eq:norm.scale}
      \|\omega\|^2_{L^2(\Lambda^p(\e M))} 
      & = \e^{n-2p}\|\omega\|^2_{L^2(\Lambda^p(M))}
        \qquad\text{and}\\
      \label{eq:ev.cale}
      \exactEV j p(\e M)
      & = \e^{-2}\exactEV 1 p(M)
    \end{align}
  \end{subequations}
\end{lem}
\begin{proof}
  The first assertion follows from the fact that we have $|w|^2_{\e^2
    g}= \e^{-2p} |w|^2_g$ and $\dvol_{\e^2g} M=\e^n\dvol_g M$
  pointwise.  The second follows from the variational characterisation
  of the $j$-th eigenvalue of Proposition~\ref{prop:abs_exact}, as we
  have the scaling behaviour
  \begin{equation*}
    \frac {\|\eta\|^2_{L^2(\Lambda^p(\e M))}}
          {\| \theta \|^2_{L^2(\Lambda^{p-1}(\e M))}}
    = \frac {\e^{n-2p} \|\eta\|^2_{L^2(\Lambda^p(M))}}
          {\e^{n-2(p-1)} \|\theta\|^2_{L^2(\Lambda^{p-1}(M))}}
    = \e^{-2} \frac {\|\eta\|^2_{L^2(\Lambda^p(M))}}
          {\|\theta\|^2_{L^2(\Lambda^{p-1}(M))}}.
  \end{equation*}
  Note that the condition $\eta = d\theta$ is \emph{independent} of
  the metric.  (This is the advantage of the characterisation of
  Proposition~\ref{prop:abs_exact}!)
\end{proof}

\section{Graph-like manifolds and their harmonic forms}
\label{sec:g-like-mfds}
%

%
\subsection{Graph-like manifolds}
\label{subsec:g-like-mfds}
%
A \emph{graph-like manifold} associated with a metric graph $X_0$ is a
family of oriented and connected $n$-dimensional Riemannian manifolds
$(\Xe)_{0<\e\leq\e_0}$ ($\e_0$ small enough) shrinking to $X_0$ as
$\e\to 0$ in the following sense. We assume that $\Xe$ decomposes as
\begin{equation}
  \label{decomp}
  \Xe=\bigcup_{e\in E}\Ee\cup\;\bigcup_{v\in V}\Ve,
\end{equation}
where $\Ve$ and $\Ee$ are called \emph{edge} and \emph{vertex
  neighbourhood}, respectively.  More precisely, we assume that $\Ve$
and $\Ee$ are closed subsets of $\Xe$ such that
\begin{equation*}
  \Ve \cap\Ee =
  \begin{cases}
    \Ye & e\in E_v\\
    \emptyset & e \notin E_v,
  \end{cases}
\end{equation*}
where $\Ye$ is a boundaryless smooth connected Riemannian manifold of
dimension $n-1$. Furthermore, we assume that $\Ve$ is $\e$-homothetic
to a fixed connected Riemannian manifold $\V$ (with metric $g_v$) as
well as $\Ye$ is $\e$-homothetic to a fixed Riemannian manifold $\Y$
(with metric $h_e$), i.e., $\Ve=\e\V$ and $\Ye=\e\Y$. Moreover, we
assume that $\Ee$ is isometric to the product $I_e\times\e\Y$.  If
$g_\e$ denotes the metric of $\Xe$ and $g_{\e,e}$, resp.\ $g_{\e,v}$,
the restriction of $g_\e$ to the edge, resp.\ vertex neighourhood, then
we have
\begin{equation}
  \label{eq:met.scale}
  g_{\e,e}=ds^2+\e^2 h_e
  \qquad\text{and}\qquad
  g_{\e,v}=\e^2 g_v
\end{equation}
(after some obvious identifications).  We often refer to a single
manifold $\Xe$ as graph-like manifold instead of the family $(\Xe)_\e$
as in the definition above.

Assume for simplicity that $\vol_{n-1} \Y=1$ for all $e \in E$ (the
general case would lead to the weighted vertex condition $\sum_{e \in
  E_v} (\vol_{n-1} \Y) \orient f_e'(v)=0$ instead
of~\eqref{eq:kirchhoff} for the metric graph Laplacian,
see~\cite{post:12,exner-post:09} for details).

We call a graph-like manifold $(X_\eps)_\eps$ \emph{transversally
  trivial} if all transversal manifolds are Moores spaces, i.e., if
$H^p(\Y)=0$ for all $1 \le p \le n-2$ and all $e \in E$.  Note that a
member of a transversally trivial graph-like manifold $X_\eps$ is not
necessarily homotopy-equivalent to the metric graph $X_0$, as the
vertex neighbourhoods need not to be contractible.
\begin{ex}
  \label{ex:gl-mfd}
  Let us construct a typical example of a transversally trivial
  graph-like manifold. Let $n \ge 2$.  For each vertex $v$ fix a
  manifold $\hat X_v$.  Remove $\deg v$ open balls from $\hat X_v$
  hence the resulting manifold $X_v$ has a boundary consisting of
  $\deg v$ many components each diffeomorphic to an $(n-1)$-sphere
  $\Sphere^{n-1}$.  For $e \in E_v$ let $Y_e= \Sphere^{n-1}$ with a
  metric such that its volume is $1$.  As (unscaled) edge
  neighbourhood, we choose $X_{1,e}:=[0,\ell_e] \times Y_e$ with the
  product metric.  Then we can construct a graph-like (topological)
  manifold $X_1$ with a canonical decomposition as in~\eqref{decomp}
  (for $\eps=1$) by identifying the $e$-th boundary component of
  $X_v$ with the corresponding end of the edge neighbourhood
  $X_{1,e}$.  By a small local change we can assume that the resulting
  manifold $X_1$ is smooth; the corresponding family of graph-like
  manifolds $(X_\eps)_{\eps>0}$ is now given as above by choosing the
  metric accordingly.
\end{ex}

\begin{rmk}
  \label{rem:gl-mfd}
  Let $X$ be a compact manifold without boundary.  The aim of the
  remark is to show that $X$ can be turned into a graph-like manifold
  with underlying metric graph being a finite tree graph: think of
  ``growing'' the tree out of the original manifold.  More formally,
  construct a graph-like manifold according to a tree graph and leave
  one cylinder of a leaf (a vertex of degree $1$) ``uncapped''; glue
  the original manifold $X$ with one disc removed together with the
  free cylinder.  Obviously, the resulting manifold is homeomorphic to
  the original manifold $X$.

  We could also modify $X$ such that it becomes a graph-like manifold
  with respect to topologically more complicated graphs.
\end{rmk}

We can now define on $\Xe$ the corresponding Hilbert spaces of
$p$-forms as in Section~\ref{sec:diff-forms}.  Since $\Xe$ has no
boundary, the formal adjoint $\delta$ of $d$ is also its Hilbert space
adjoint.  Moreover, the Hodge Laplacian on $p$-forms on $\Xe$ is given
by $\Delta^p_{\Xe}=dd^*+d^*d$ where $d=d_{\Xe}$ and $d^*=d_{\Xe}^*$
are the classical exterior derivative and co-derivative on a manifold
(as unbounded operators in the $L^2$-spaces).

\subsection{Harmonic forms}
\label{sec:harmonic}

For completeness we finally turn to the dimension of the class of
harmonic $1$-forms.

For the graph, this dimension is given by its first Betti number,
i.e., $b_1(X_0)=|E|-|V|+1$, while for the manifold $\Xe$ it is given
by the dimension of its first cohomology group $H^1(\Xe)$. Since $\Xe$
arises from the graph $X_0$, the dimension of $H^1(\Xe)$ is the sum of
$b_1(X_0)$ and the dimension of a subset of the first cohomology group
of $\bigcup_{v\in V}\Ve$, meaning that the graph-like manifold
inherits part of the topology of the underlying metric graph.

In particular, if $\Y$ has trivial $p$-th cohomology group for $1\leq
p\leq n-2$, i.e., $H^p(\Y)=0$ for all $e \in E$, then the cohomology
groups of $\Xe$ can be computed explicitly. Using Mayer-Vietoris
sequence, the natural splitting~\eqref{decomp} and Poincar\'{e}
duality, we obtain
\begin{equation*}
  H^k(\Xe) =
  \begin{cases}
    \R & k \in \{0,n\}\\
    \bigoplus_{v\in V}H^1(\V)\oplus H^1(X_0) \quad& k \in \{1,n-1\}\\
    \bigoplus_{v\in V} H^k(\V) & k \in \{2,\ldots,n-2\}
  \end{cases}
\end{equation*}
For the general case, i.e., when some or all of the $\Y$ have
non-trival $p$-th cohomology groups for $1\leq p\leq n-2$, we do not
have a general formula.  However, again using Mayer-Vietoris sequence,
it is possible to compute the cohomology groups explicitly for
concrete examples of edge and vertex neighbourhoods.

\section{Proof of the main theorem}
\label{sec:proof}
%

Let us now prove the main result of this article.  The convergence
result of our main theorem, i.e.,~\eqref{eq:ev-0-forms} of
Theorem~\ref{thm:main}, is more or less ``trivial'' in the sense that
it follows from previous convergence results for functions by a simple
supersymmetry argument.

The divergence~\eqref{eq:ev-1-forms} of Theorem~\ref{thm:main} is new and
proven in Subsection~\ref{sec:co-exact}.  As preparation, we need some
estimates of exact $p$-forms with absolute boundary conditions on the
building blocks of our graph-like manifolds provided in
Subsection~\ref{subsec:vertex-edge}.

\subsection{Convergence for exact 1-forms}
\label{sec:exact_forms}

Let $\Xe$ be a compact graph-like manifold as constructed in
Section~\ref{subsec:g-like-mfds} associated with a metric graph $X_e$.
We have already noticed
in~\eqref{eq:trivial-ev-rel}--\eqref{eq:trivial-ev-rel-graph} and in
Remark~\ref{rem:supersymmetry}, that the co-exact $1$-form eigenvalues
equal the (exact) $0$-form eigenvalues, i.e., the eigenvalues of the
Laplacian on functions on $X_0$ and $\Xe$.  For the functions, we have
the following result, first proven in the manifold case
in~\cite{exner-post:05} (based on the
results~\cite{kuchment-zeng:01,rubinstein-schatzman:01}).  For a
detailed overview and detailed proofs of the result, we refer
to~\cite{post:12}.

Denote by $\lambda_j(\Xe)$ and $\lambda_j(X_0)$ the eigenvalues (in
increasing order, repeated according to their multiplicity) of the
Laplacian acting on functions on the manifold and the metric graph
(see~\eqref{eq:kirchhoff} for the metric graph Laplacian).
\begin{prop}[\cite{exner-post:05, post:12}]
  \label{prp:exner-post}
  Let $\Xe$ be a compact graph-like manifold associated with a metric graph
  $X_e$. Then we have
  \begin{equation*}
    |\lambda_j(\Xe) - \lambda_j(X_0)| 
    = \Err{\e^{1/2}/\ell_0}
    \qquad\text{for all $j=1,2,\dots$.}
  \end{equation*}
  where $\ell_0=\min_e \{\ell_e,1\}>0$ denotes the minimal edge length.
  Moreover, the error depends only on $j$, and the building blocks
  $X_v$, $Y_e$ of the graph-like manifold.
\end{prop}
We will need the precise dependency on the edge length and other
parameters in Section~\ref{sec:examples} when considering
\emph{families} of metric graphs and graph-like manifolds.  The exact
statement on the error term follows from a combination of Thms.~6.4.1,
7.1.2 and~4.6.4 of \cite{post:12}.

Denote by $\exactEV j 1 (\Xe)$ and $\exactEV j 1 (X_0)$ the $j$-th
eigenvalues of the exact $1$-form Laplacian on $\Xe$ and $X_0$,
respectively.  The above-mentioned convergence for the eigenvalues for
functions immediately gives the convergence for exact $1$-forms, using
a simple supersymmetry argument as in~\cite[Sec.~1.2]{post:09c}.  For
the convenience of the reader, we give a simple proof here:

\begin{thm}
  \label{thm:exact}
  Let $\Xe$ be a graph-like manifold with underlying metric graph
  $X_0$.  Denote by $\exactEV j 1(\Xe)$ and $\exactEV j 1(X_0)$ the
  $j$-th exact $1$-form eigenvalue on $\Xe$ and $X_0$,
  respectively. Then
  \begin{equation*}
    \exactEV j 1(\Xe) \underset{\e\to 0}{\longrightarrow}
    \exactEV j 1(X_0)
    \qquad\text{for all $j=1,2,\dots$.}
  \end{equation*}
\end{thm}
\begin{proof}
  We will just show that the eigenspaces for non-zero eigenvalues of
  $\Delta_{\Xe}^1=\Delta^1=d d^*$ and $\Delta_{\Xe}^0=\Delta^0=d^* d$
  are isomorphic (the argument works for $\e>0$ and $\e=0$ as well).
  The convergence result then follows immediately from
  Proposition~\ref{prp:exner-post}.
  
  As isomorphism, we choose
  \begin{equation*}
    d \colon \ker(\Delta^0 - \lambda) 
    \longrightarrow
    \ker(\Delta^1 - \lambda)
  \end{equation*}
  for $\lambda \ne 0$.  First, note that if $f \in \ker(\Delta^0 -
  \lambda)$, then
  \begin{equation*}
    \Delta^1 df
    = d d^* d f
    = d \Delta^0 f
    = \lambda f,
  \end{equation*}
  i.e., $df \in \ker(\Delta^1-\lambda)$, hence the above map is
  properly defined.  The map $d$ as above is injective: If $df=0$ for
  $f \in \ker (\Delta^0-\lambda)$ then $\lambda f=\Delta^0 f =
  d^*df=0$.  As $\lambda \ne 0$ we have $f=0$.  For the surjectivity,
  let $\alpha \in \ker(\Delta^1-\lambda)$. Set $f:=\lambda^{-1} d^*
  \alpha$ (we use again that $\lambda \ne 0$).  Then
  \begin{equation*}
    d f = d (\lambda^{-1} d^* \alpha) =\lambda^{-1} \Delta^1 \alpha = \alpha,
  \end{equation*}
  i.e., $d$ as above is surjective.  In particular, we have shown that
  the spectrum of $\Delta^0$ and $\Delta^1$ away from $0$ is the same,
  including multiplicity.
\end{proof}

\subsection{Eigenvalue asymptotics on the building blocks}
\label{subsec:vertex-edge}

We will now provide some eigenvalue asymptotics for eigenvalues of
exact $p$-forms with absolute boundary conditions on the building
blocks of our graph-like manifold.  These asymptotics are needed in
order to make use of the eigenvalue estimate from below of
Proposition~\ref{prop:gentile-pagliara:95}.

A vertex neighbourhood $\Ve$ is by definition $\e$-homothetic, i.e,
$\Ve=\e X_v$.  As a result of Lemma~\ref{lem1} we have:
\begin{cor}
  \label{prop:vertex}
  Let $\Ve$ be a vertex neighbourhood of a graph-like
  manifold $\Xe$. Then, the smallest positive eigenvalue of the
  Laplacian acting on exact $p$-forms on $\Ve$ with absolute boundary
  conditions satisfies:
  \begin{equation}
    \label{eq2}
    \exactEV 1 p(\Ve)=\e^{-2}\exactEV 1 p(\V).
  \end{equation}
\end{cor}

For the edge neighbourhood, we have to work a little bit more.  Note
that we cannot make use of the product structure of the underlying
space as the product does not respect exact and co-exact forms.
\begin{prop}
  \label{prop:edge}
  Let $\Ee$ be an edge neighbourhood of a $n$-dimensional graph-like
  manifold $\Xe$. Then, the smallest eigenvalue of the Laplacian
  acting on exact $p$-forms ($2 \le p \le n-1$) with absolute boundary
  conditions satisfies
  \begin{equation}
    \label{eq7}
    \exactEV 1 p(\Ee)=\e^{-2}c_p(\e),
  \end{equation}
  where $c_p(\e)\to \exactEV 1 p(\Y)>0$ as $\e\to 0$, and where
  $\exactEV 1 p(\Y)$ denotes the first eigenvalue of the Laplacian
  acting on exact $p$-forms on $\Y$.
\end{prop}
\begin{proof}
  By Proposition~\ref{prop:abs_exact} we have to analyse the quotient
  $\|\eta\|^2/\|\theta\|^2$ for an exact $p$-form $\eta$ and a
  $(p-1)$-form $\theta$ such that $\eta=d \theta$.
  Recall that $\Ee=I_e\times\e\Y$ (i.e., $I_e \times \Y$ with metric
  $g_{\e,e}=ds^2+\e^2h_e$).  Then, the $(p-1)$-form $\theta$ on $\Ee$
  can be written as
  \begin{equation}
    \label{eq4}
    \theta=\theta_1\wedge ds+\theta_2
  \end{equation}
  where $\theta_1$ resp.\ $\theta_2$ is a $(p-2)$-form resp.\
  $(p-1)$-form on $\Y$.  Using the scaling behaviour of the metric in
  a similar way as Lemma~\ref{lem1}, we have
  \begin{equation}
    \label{eq5}
    \begin{split}
      \|\theta\|^2_{L^2(\Lambda^{p-1}(\Ee))}
      &= \int_{\Ee}|\theta|_{g_{\e,e}}^2\dvol\Ee \\
      &=\int_{I_e}\int_{\Y}
         \bigl(
            \e^{-2(p-2)}|\theta_1|_{h_e}^2
            +\e^{-2(p-1)}|\theta_2|_{h_e}^2
         \bigr) \e^{n-1}\dd s\dvol\Y\\
      & =\e^{n-2p+1} \int_{I_e} \int_{\Y}
         \bigl(
            \e^2|\theta_1|_{h_e}^2
            +|\theta_2|_{h_e}^2
         \bigr) \dd s \dvol \Y,
    \end{split}
  \end{equation}
  where the $\e$-factors appears due to the scaled metric $\e^2 h_e$.
  The decomposition of $d\theta$ according to~\eqref{eq4} is given by
  \begin{equation}
    \label{dtheta}
    d\theta 
    =(d_{\Y}\theta_1+\partial_s\theta_2)\wedge \dd s+d_{\Y}\theta_2,
  \end{equation}
  hence
  \begin{equation}
    \label{eq6}
    \begin{split}
      \|d \theta\|^2_{L^2(\Lambda^p(\Ee))}
      &= \int_{\Ee}|d \theta|_{g_{\e,e}}^2\dvol\Ee \\
      &= \int_{I_e}\int_{\Y}
           \bigl(
              \e^{-2(p+1)}|\theta_1+\partial_s \theta_2|_{h_e}^2
              +\e^{-2p}|d_{\Y}\theta_2|_{h_e}^2
           \bigr)
           \e^{n-1} \dd s \dvol\Y\\
      & = \e^{n-2p-1} \int_{I_e}\int_{\Y} 
           \bigl(
             \e^2|\theta_1+\partial_s \theta_2|_{h_e}^2
             +|d_{\Y}\theta_2|_{h_e}^2
           \bigr)
          \dd s\dvol\Y.
    \end{split}
  \end{equation}
  In particular, if we substitute~\eqref{eq5} and~\eqref{eq6} into the
  quotient $\|\eta\|^2/\|\theta\|^2$ we conclude
  \begin{equation*}
    \frac {\|d \theta\|^2_{L^2(\Lambda^p(\Ee))}}
          {\; \|\theta\|^2_{L^2(\Lambda^{p-1}(\Ee))}}
    =\e^{-2} \dfrac{\int_{I_e}\int_{\Y} 
                       \bigl(
                         \e^2|\theta_1+\partial_s \theta_2|_{h_e}^2
                         +|d_{\Y}\theta_2|_{h_e}^2 
                       \bigr)
                    \dd s\dvol\Y}
                  {\int_{I_e}\int_{\Y}
                       \bigl(
                         \e^2|\theta_1|_{h_e}^2+|\theta_2|_{h_e}^2
                       \bigr)
                     \dd s \dvol\Y}
  \end{equation*}
  In particular, together with Proposition~\ref{prop:abs_exact} this
  yields
  \begin{equation*}
    \exactEV 1 p(\Ee)
    =\e^{-2}c_p(\e)
  \end{equation*}
  with
  \begin{equation*}
    c_p(\e)
    = \sup
      \BIGset{
        \dfrac{\int_{I_e}\int_{\Y}
                 \e^2\bigl(
                   |\theta_1+\partial_s \theta_2|_{h_e}^2
                   +|d_{\Y}\theta_2|_{h_e}^2
                 \bigr)
                \dd s\dvol\Y}
              {\int_{I_e}\int_{\Y}
                 \bigl(
                   \e^2 |\theta_1|_{h_e}^2 +|\theta_2|_{h_e}^2
                 \bigr)
                \dd s \dvol\Y}
              }
      { \text{\parbox{20ex}
           {$\theta=\theta_1 \wedge \dd s + \theta_2 \ne 0$,\newline
            $\theta_1$ $(p-2)$-form,\newline $\theta_2$ $(p-1)$-form}}}.
  \end{equation*}
  In the limit $\e \to 0$, this constant tends to a number $c_p(0)$ given by
  \begin{equation*}
    c_p(0)
    = \sup
      \BIGset{
        \dfrac{\int_{I_e}\int_{\Y} |d_{\Y}\theta_2|_{h_e}^2 \dd s\dvol\Y}
              {\int_{I_e}\int_{\Y} |\theta_2|_{h_e}^2 \dd s \dvol\Y}        
             }
             {\text{$\theta_2 \ne 0$ $(p-1)$-form}}.
  \end{equation*}
  This constant is the min-max characterisation of the first
  eigenvalue of the operator $\id \otimes \exactLapl \Y p$ acting on
  $L^2(I_e) \otimes L^2(\Lambda^p(\Y))$, whose spectrum agrees with
  the one of $\exactLapl \Y p$ (see
  e.g.~\cite[Thm.~XIII.34]{reed-simon-4}).  Hence, we have
  $c_p(0)=\exactEV 1 p(\Y)$.
\end{proof}

\subsection{Divergence for co-exact forms}
\label{sec:co-exact}

We will assume for the rest of this section that $n \ge 3$.  If
$\dim\Xe=2$, then the spectrum of exact and co-exact $1$-forms
coincide by duality.  Hence, the spectrum of the Hodge Laplacian is
entirely determined by the spectrum on functions, and hence its
behaviour is covered by the results of
Subsection~\ref{sec:exact_forms}.

We come now to the proof of the divergence of our main theorem, namely
to~\eqref{eq:ev-1-forms} of Theorem~\ref{thm:main}.  We will make use
of Proposition~\ref{prop:gentile-pagliara:95} for $2 \le p \le n-1$
assuming that $H^{p-1}(\Y)=0$ for all $e\in E$.  Then, we will
briefly explain how the same argument works for $p=2$ and non-trivial
cohomology $H^1(\Y)\neq 0$ for some $e\in E$ using
Proposition~\ref{prop:mcgowan}.

Let $H^1(\Y)=0$ for all $e\in E$. Let
\begin{equation*}
  \mathcal{U}_{\e}=\{\Uv\}_{v\in V}\cup\{\Ee\}_{e\in E}
\end{equation*}
be an open cover of $\Xe$, where $\Uv$ is the open $\e$-neighbourhood
of $\Ve$ in $\Xe$, or in other words, a slightly enlarged vertex
neighbourhood $\Ve$ to ensure that $\mathcal U_\e$ is an open cover.

It is easily seen that $\mathcal U_\e$ has intersection up to degree
$2$ only (no three different sets of $\mathcal U_\e$ have non-trivial
intersection).  The intersections of degree $2$ are given by $\Xve=\Uv
\cap \Ee$ which is empty ($e \notin E_v$) or otherwise isometric to
the product $(0,\e) \times\Ye$, hence $\e$-homothetic with the product
$(0,1) \times \Y$ (recall that we enlarged $\Ve$ by an
$\e$-neighbourhood).  Moreover, $\Xve$ is homeomorphic to $(0,1)\times
\Y$, and hence homotopy-equivalent with $\Y$.  In particular,
$H^{p-1}(\Xve)=H^{p-1}(\Y)$.

Recall that $\exactEV j p(\Xe)$ denotes the $j$-th exact $p$-form
eigenvalue on $\Xe$, which equals the $j$-th co-exact
$(p-1)$-eigenvalue $\coexactEV j {p-1}(\Xe)$.  We assume $n \ge 3$, as
in dimension $2$ the Hodge Laplace spectrum is entirely determined by
the scalar case.  Denote by $H^p(\Y)$ the $p$-th cohomology group of the
transversal manifold $\Y$ of the edge neighbourhodd $\Ee$.
\begin{thm}
  \label{thm3}
  Let $\Xe$ be a graph-like manifold of dimension $n \ge 3$ with
  underlying metric graph $X_0$.  Assume that $2 \le p \le n-1$ and
  that the $(p-1)$-th cohomology group of the transversal manifold
  $\Y$ vanishes for all $e \in E$, i.e., $H^{p-1}(\Y)=0$.  Then, the
  first eigenvalue of the Laplacian acting on exact $p$-forms on $\Xe$
  satisfies
  \begin{equation*}
    \exactEV 1 p (\Xe)
    \ge \tau_p \e^{-2},
  \end{equation*} 
  where $\tau_p>0$ is a constant depending only on the building blocks
  $X_v$ and $Y_e$ of the graph-like manifold, the minimal length
  $\ell_0=\min_{e \in E}\{\ell_e,1\}$ and $p$.  In particular, all
  eigenvalues $\exactEV j p (\Xe)$ of exact $p$-forms and all
  eigenvalues $\coexactEV j {p-1} (\Xe)$ of co-exact $(p-1)$-forms
  tend to $\infty$ as $\e \to 0$.
\end{thm}
\begin{proof}
  We will apply Proposition~\ref{prop:gentile-pagliara:95} to the manifold $\Xe$
  and the cover $\mathcal U_\e$ (having no intersection of degree
  higher than $2$).  The assumptions on the cohomology are fulfilled
  as the $(p-1)$-th cohomology of the intersections of degree $2$ of
  the cover vanishes (as we have already stated above).  We first look
  at the denominator of the right hand side of the estimate in
  Proposition~\ref{prop:gentile-pagliara:95} and obtain in our situation here
  \begin{align*}
     &\sum_{v \in V}
       \Bigg(
         \frac 1 {\exactEV 1 p(\Ve)}+
          \sum_{e \in E_v}
         \bigg(\frac{c_{n,p} \|d \rho_\e\|^2_\infty}
                   {\exactEV 1 {p-1}(\Xve)}+1
         \bigg)
         \bigg(\frac 1 {\exactEV 1 p(\Ve)}
             +\frac 1 {\exactEV 1 p(\Ee)}
        \bigg)
      \Bigg)\\
     &\qquad\qquad + \sum_{e \in E}
       \Bigg(
          \frac 1 {\exactEV 1 p(\Ee)}
         +\sum_{v=\partial_\pm e}
         \bigg(\frac{c_{n,p} \|d \rho_\e\|^2_\infty}
                   {\exactEV 1 {p-1}(\Xve)}+1
         \bigg)
         \bigg(\frac 1 {\exactEV 1 p(\Ve)}
             +\frac 1 {\exactEV 1 p(\Ee)}
        \bigg)
      \Bigg)\\
     =&\sum_{v \in V}
       \Bigg(
         \frac 1 {\exactEV 1 p(\Ve)}
         + \frac {\deg v} {\exactEV 1 p(\Ee)}
         +2 \sum_{e \in E_v}
         \bigg(\frac{c_{n,p} \|d \rho_\e\|^2_\infty}
                   {\exactEV 1 {p-1}(\Xve)}+1
         \bigg)
         \bigg(\frac 1 {\exactEV 1 p(\Ve)}
             +\frac 1 {\exactEV 1 p(\Ee)}
        \bigg)
      \Bigg)\\
     =&
        \e^2 \sum_{v \in V}
        \Bigg(
          \frac 1 {\exactEV 1 p(\V)}
          + \frac {\deg v} {c_p(\e)}
          +2 \sum_{e \in E_v}
          \bigg(\frac{c_{n,p} \e^2 \|d \rho_\e\|^2_\infty}
                    {\exactEV 1 {p-1}(X_{v,e})} +1
          \bigg)
          \bigg(\frac 1 {\exactEV 1 p(\V)}
              +\frac 1 {c_p(\e)}
         \bigg)
       \Bigg)
     =: \e^2 C_p(\e).
  \end{align*}
  Note that the cover $\mathcal U_\e$ is labelled by $v \in V$ and $e
  \in E$, and we have rewritten the sum over the edges as a sum over
  the vertices (leading to the extra term with $\deg v$ and the factor
  $2$) for the second equality.  For the third equality, we have used
  the scaling behaviour of the eigenvalues in equations~\eqref{eq2}
  and~\eqref{eq7}, and a similar one for the $\e$-homothetic overlap
  manifold $\Xve$.
  
  Let us now analyse the constant $C_p(\e)$ as $\e \to 0$. First, we
  have seen in Proposition~\ref{prop:edge} that $c_p(\e)\to \exactEV 1
  p (\Y)>0$.  Moreover, the norm of the derivative of the partition of
  unit norm depends on $\e$ as these functions have to change from $0$
  to $1$ on a length scale of order $\e$ on the vertex neighourhoods
  and on a length scale of order $\ell_0$ on the edge neighourhood,
  hence the derivative is of order $\e^{-1}+\ell_0^{-1}$ and $\e^2 \|d
  \rho_\e\|_\infty^2 = \Err{1}+\Err{(\e/\ell_0)^2}$ (we will need the
  dependency on $\ell_0$ for Section~\ref{sec:examples} when we
  allow $\ell_0$ also to depend on $e$).  In particular, $C_p(\e) \to
  C_p(0)$ as $\e \to 0$ provided $\e/\ell_0$ remains bounded, where
  $C_p(0)$ depends only on some data of the building blocks.
  
  Proposition~\ref{prop:gentile-pagliara:95} now gives
  \begin{equation*}
    \exactEV 1 p (\Xe)
    \geq \frac{2^{-3}}{\e^2 C_p(\e)}
  \end{equation*}
  which proves our assertion.
\end{proof}

Removing the assumption of vanishing cohomology groups we have the
following theorem whose proof follows the line of the previous one
where we use Proposition~\ref{prop:mcgowan} for the estimate of a higher
eigenvalue for exact $p$-forms on $\Xe$.

\begin{thm}
  \label{prop2}
  Let $\Xe$ be a graph-like manifold of dimension $n \ge 3$ with
  underlying metric graph $X_0$. Then the $N$-th eigenvalue of the
  Laplacian acting on exact $p$-forms on $\Xe$ satisfies
  \begin{equation*}
    \exactEV N p (\Xe)
    \ge \wt \tau_p \e^{-2},
  \end{equation*} 
  where $\wt \tau_p>0$ is as before and where
  \begin{equation*}
    N = 1 + \sum_{v\in V}\sum_{e\in E_v} \dim H^{p-1}(\Y)
    = 1 + 2 \sum_{e \in E} \dim H^{p-1}(\Y).
  \end{equation*}
\end{thm}

\begin{rmk}
  \label{rmk:eigenvalues}
  We point out that the first $N-1$ eigenvalues of the Theorem above
  are strictly positive since we consider the spectrum away from
  zero. However, it is an open question how the eigenvalues behave
  asymptotically as $\e \to 0$.
\end{rmk}

\section{Examples}
\label{sec:examples}
%

Let us discuss some consequences of our asymptotic description of the
Hodge Laplacian spectrum.

\subsection{Hausdorff convergence of the spectrum and spectral gaps}
\label{sec:hausdorff}
Let us first come to Corollary~\ref{cor:main.conv}, the Hausdorff
convergence of the entire Hodge Laplace spectrum.  Let $A,B \subset
\R$ be two compact sets.  The \emph{Hausdorff distance} of $A$ and $B$
is defined as
\begin{equation}
  \label{eq:hausdorff.dist}
  d(A,B) := \max \{ \sup_{a \in A} d(a,B), \sup_{b \in B} d(b,A)\},
  \quad\text{where}\quad
  d(a,B):=\inf_{b \in B} |a-b|.
\end{equation}
A sequence $(A_n)_n$ of compact sets $A_n \subset \R$ \emph{converges
  in Hausdorff distance} to $A_0$ if and only if $d(A_n,A) \to 0$ as
$n \to \infty$.  In particular, $d(A_n,A) \to 0$ if and only if for
all $\lambda_0 \in A_0$ there exists $\lambda_\e \in A_\e$ such that
$|\lambda_0-\lambda_\e| \to 0$ and for all $x \in \R \setminus A_0$
there exists $\eta>0$ such that $[x-\eta,x+\eta] \cap A_\e =
\emptyset$ for $\e$ sufficiently small (see
e.g.~\cite[Prp.~A.1.6]{post:12}).

Corollary~\ref{cor:main.conv} about spectral convergence is now an
immediate consequence of Theorem~\ref{thm:main}, as in a compact
interval $[0,\lambda_0]$, eventually all divergent eigenvalues from
higher forms leave this interval, and the remaining ones converge.

A spectral gap of an operator $\Delta \ge 0$ is a non-empty interval
$(a,b)$ such that
\begin{equation*}
  \spec \Delta \cap (a,b)=\emptyset.
\end{equation*}
Corollary~\ref{cor:main} on spectral gap is again an immediate
consequence of the Hausdorff convergence of
Corollary~\ref{cor:main.conv} under the assumption that the manifold
is transversally trivial (i.e., all transversal manifolds $Y_e$ have
trivial $p$-th cohomology for all $1 \le p \le n-2$).

Examples of manifolds with spectral gaps can be generated in different
ways.  In~\cite{post:03a,lledo-post:08} we constructed (non-compact)
abelian covering manifolds having an arbitrary large number of gaps in
their essential spectrum of the scalar Laplacian, and
in~\cite{acp:09}, we extended the analysis to the Hodge Laplacian on
certain manifolds.

One can construct metric graphs with spectral gaps (and hence
graph-like manifolds with spectral gaps) with a technique called
\emph{graph decoration} that works as follows. We consider a finite
metric graph $X_0$ and a second finite metric graph $\wt X_0$.  For
each $v \in V(X_0)$, let $\wt X_0 \times \{v\}$ be a copy of a finite
metric graph $\wt X_0$.  Fix a vertex $\wt v$ of $\wt X_0$.  Then the
graph decoration of $X_0$ with the graph $\wt X_0$ is the graph
obtained from $X_0$ by identifying the vertex $\wt v$ of $\wt X_0
\times \{v\}$ with $v$.  This decoration opens up a gap in the
spectrum of the Laplacian on function on $X_0$ as described
in~\cite{kuchment:05} and therefore in its $1$-form
Laplacian. Consequently, the associated graph-like manifold has a
spectral gap in its $1$-form Laplacian (and no spectrum away from $0$
for higher forms due to the divergence).

More examples of \emph{families} of graphs and their graph-like
manifolds with spectral gaps are given in
Section~\ref{sec:fam.graphs}.

\subsection{Manifolds with special spectral properties}
\label{sec:spec.prop}

Let $(X_\eps)_{\eps>0}$ be a graph-like manifold constructed from a
metric graph $X_0$ with underlying graph $(V,E,\partial)$.  We assume
that the graph-like manifold is transversally trivial, i.e., all
transversal manifolds $Y_e$ have trivial homlogy $H^p(Y_e)=0$ for all
$1 \le p \le n-2$.  An example of a construction of such graph-like
manifolds is given in Example~\ref{ex:gl-mfd}.

For simplicity, we assume that $X_0$ is equilateral, i.e., all edge
lengths are given by a number $\ell>0$.  (One can easily extend the
results to the case when $c_- \ell \le \ell_e \le c_+ \ell$ for all $e
\in E$ and some constants $c_\pm>0$.)

We write 
\begin{equation}
  \label{eq:notation}
  a_\eps \lesssim b_\eps, \qquad
  a_\eps \gtrsim b_\eps, \qquad
  a_\eps \simeq b_\eps
\end{equation}
if 
\begin{equation}
  \label{eq:notation'}
  \tag{\ref{eq:notation}'}
  a_\eps \le \const_+ b_\eps, \qquad
  a_\eps \ge \const_- b_\eps, \qquad
  \const_- a_\eps \le b_\eps \le \const_+ a_\eps
\end{equation}
for all $\eps>0$ small enough and constants $\const_\pm$
\emph{independent} of $\eps$.

Let us first summarise the asymptotic spectral behaviour of a
graph-like manifold $X_\eps$ and its dependence on the parameters
$\eps$, $\ell$, $|V|$, and $|E|$.  In particular, we have for
$0$-forms, exact $p$-forms and co-exact $(p-1)$-forms and the volume:
\begin{align}
  \label{eq:ev.fcts}
  |\lambda_j^0(X_\eps)-\lambda_j^0(X_0)| &\lesssim
  \dfrac{\eps^{1/2}}{\ell_0}& (\ell_0=\min\{\ell,1\})\\
  \label{eq:ev.forms}
  \exactEV 1 p(X_\eps) = \coexactEV 1 {p-1}(X_\eps)
  &\gtrsim \dfrac 1{\eps^2|E|(1+\eps^2/\ell^2)}\\
  \label{eq:volume}
  \vol X_\eps & \eqsim
  \eps^n|V| + \eps^{n-1} \ell |E|,
\end{align}
where the constants in $\lesssim$ etc.\ depend only on the building
blocks $X_v$ and $Y_e$ of the (unscaled, i.e, $\eps=1$) graph-like
manifold.  Eq.~\eqref{eq:ev.forms} follows from analysing the lower
bound constant $\tau_p$ in Theorem~\ref{thm3} (or
Theorem~\ref{prop2}). We see that the constant $C_p(\e)$ in its proof
is bounded from above by
\begin{equation*}
  C_p(\e) \lesssim \bigl(|V|+|E|(1+\e^2/\ell^2)\bigr)
  \lesssim |E|(1+\e^2/\ell^2)
\end{equation*}
where again the constants in $\lesssim$ depend only on the building
blocks and where we used $|V| \le \sum_{v \in V} \deg v = 2|E|$ for
any graph $G$ (assuming that there are no isolated vertices, i.e.,
vertices of degree $0$).

Let us now assume that $\ell=\ell_\eps=\eps^\gamma$ depends on $\eps$
for some $\gamma \in \R$ (negative $\gamma$'s are not excluded).  In
particular, $X_0$ now also depends on $\eps$, and we write
$\eps^\gamma X_0$ for metric graph with all edge lengths mulitplies by
$\eps^\gamma$.  We have
\begin{itemize}
\item
  \begin{itemize}
  \item For the closeness in~\eqref{eq:ev.fcts} to hold we need
    $\gamma<1/2$, as the error term is of order
    $\eps^{1/2}/\min\{\eps^\gamma,1\}=\eps^{1/2-\max\{\gamma,0\}}$.

  \item For the metric graph eigenvalue, we have
    $\lambda_j(\eps^\gamma X_0)=\eps^{-2\gamma} \lambda_j(X_0)$.
  \item For the metric graph eigenvalue (of order $\eps^{-2\gamma}$)
    to be dominant with respect to the error (of order
    $\eps^{1/2-\max\{\gamma,0\}}$), we need $\gamma>-1/4$.  Hence
    \begin{equation}
      \label{eq:ev.fcts'}
      \tag{\ref{eq:ev.fcts}'}  
      \lambda_j^0(X_\eps)
      \begin{cases}
        \eqsim \eps^{-2\gamma},& -1/4 < \gamma \;({<}\, 1/2),\\
        \lesssim \eps^{1/2}, & \gamma \le -1/4.
      \end{cases}
    \end{equation}
  \end{itemize}
  
\item For the divergence in~\eqref{eq:ev.forms} to hold we need
  $\gamma < 2$.  In particular, we have
  \begin{equation}
    \label{eq:ev.forms'}
    \tag{\ref{eq:ev.forms}'}
    \exactEV 1 p(X_\eps) \gtrsim
    \begin{cases}
      \eps^{-2},& \gamma \le 1,\\
      \eps^{-4+2\gamma},& 1 \le \gamma \;({<}\, 2).
    \end{cases}
  \end{equation}
  
\item For the volume, we have
  \begin{equation}
    \label{eq:volume'}
    \tag{\ref{eq:volume}'}
    \vol X_\eps \eqsim \eps^n|V|+\eps^{n-1+\gamma}|E|
    \eqsim
    \begin{cases}
        \eps^{n-1+\gamma}|E|,&  \gamma \le 1,\\
        \eps^n|V|, & \gamma \ge 1.
      \end{cases}
  \end{equation}
\end{itemize}

\paragraph{Constant volume and arbitrarily large form eigenvalues:}
The following example gives another answer to a question of
Berger~\cite{berger:73}, answered already
in~\cite{gentile-pagliara:95} (see also the references therein for
further contributions).  Their Theorem~1 says that for any closed
manifold $X$ of dimension $n\ge 4$ there exits a metric of volume $1$
such that $\lambda_1^p(X)$ (the non-harmonic spectrum) is arbitrarily
large.  Note that their construction corresponds to a simple graph
with one edge and two vertices.  We have the following result.
\begin{prop}
  \label{prp:ex.g-p}
  On any transversally trivial graph-like manifold of dimension $n \ge
  3$ there exists a family of metrics $\wt g_\eps$ of volume $1$ such
  that for the first eigenvalue on exact $p$-forms we have
  \begin{equation*}
    \exactEV 1 p(X,\wt g_\eps)
    \to \infty \qquad
    \text{as $\eps \to 0$}
  \end{equation*}
  for $2 \le p \le n-1$.  Moreover, the function ($p=0$) and exact
  $1$-form spectrum converges to $0$, i.e.,
  \begin{equation*}
    \lambda_1^0(X,\wt g_\eps) = \exactEV 1 1(X,\wt g_\eps)
    \to 0 \qquad
    \text{as $\eps \to 0$.}
  \end{equation*}
\end{prop}
\begin{proof}
  Let $g_\eps$ be the metric of the graph-like manifold as constructed
  in Section~\ref{subsec:g-like-mfds}.  For any $\gamma<1$, we have
  \begin{equation*}
    \exactEV 1 p(X,g_\eps) (\vol (X,g_\eps))^{2/n}
      \gtrsim \eps^{-2}\eps^{2(n-1+\gamma)/n}
      =\eps^{-2(1-\gamma)/n} \to \infty \qquad
      \text{as $\eps \to 0$}
  \end{equation*}
  by~\eqref{eq:ev.forms'} and \eqref{eq:volume'}.  Set now $\wt g_\eps
  := \vol (X,g_\eps)^{-2/n}g_\eps$, then $\vol(X,\wt g_\eps)=1$ and
  \begin{equation*}
    \exactEV 1 p(X,\wt g_\eps)
    = \vol(X,g_\eps)^{2/n} \exactEV 1 p (X, g_\eps)
    \gtrsim \eps^{-2(1-\gamma)/n} \to \infty \qquad
    \text{as $\eps \to 0$.}
  \end{equation*}
  If $-1/4 < \gamma<1/2$, then the $0$-form (and exact $1$-form)
  eigenvalues of the metric graph and the manifold are close and
  $\lambda_j^0(X,g_\eps) \eqsim \eps^{-2\gamma}$, hence
  \begin{equation*}
    \lambda_j^0(X,\wt g_\eps)
    = \vol (X,g_\eps)^{2/n} \lambda_j^0(X,g_\eps)
    \eqsim \eps^{2(n-1+\gamma)/n}\eps^{-2\gamma}
    =\eps^{2(n-1)(1-\gamma)/n} \to 0.\qedhere
  \end{equation*}
\end{proof}
The transversal length scale (the one of the transversal manifolds
$Y_e$) is $\eps^{(1-\gamma)/n} \to 0$, while the longitudinal length
scale (the one of the metric graph edges $I_e$) is
$\eps^{-(1-1/n)(1-\gamma)} \to \infty$ as $\eps \to 0$.

Unfortunately, we cannot extend the result
of~\cite{gentile-pagliara:95} to the case $n=3$ and $1$-forms (as the
exact $1$-form spectrum \emph{converges}).

\subsection{Families of manifolds with special spectral properties}
\label{sec:fam.graphs}

Let us now consider families of graph-like manifolds, constructed
according to a sequence of graphs $\{G^i\}_{i \in \N}$.  We assume for
simplicity that the vertex degree is uniformly bounded, say by $k_0
\in \N$.  Then we have (if there are no isolated vertices)
\begin{equation*}
  \abs{V(G^i)} \le  \sum_{v \in V(G^i)} \deg_{G^i} v
  = 2 \abs{E(G^i)} \le 2k_0 \abs{V(G^i)},
\end{equation*}
i.e., $\nu_i:=\abs{V(G^i)} \simeq \abs{E(G^i)}$ as $i \to \infty$.  We first
start with a general statement about the spectral convergence. Assume
that $\{G^i\}_{i \in \N}$ is a family of discrete graphs and that
$\{X_0^i\}_{i \in \N}$ is the family of associated equilateral metric
graphs, each graph $X_0^i$ having edge lengths equal to $\ell_i$.

Assume now that we construct accordingly a family of graph-like
manifolds $\{\Xei\}_{i\in \N}$ where the building blocks $X_v$ and
$Y_e$ are isometric to a given number of prototypes (independent of
$i$), such that $Y_e$ all have trivial cohomology for all $1 \le p \le
n-2$ (see Example~\ref{ex:gl-mfd}), so that all graph-like manifolds
$\Xei$ are transversally trivial and hence our
estimates~\eqref{eq:ev.fcts}--\eqref{eq:volume} are uniform in the
building blocks and~\eqref{eq:ev.forms} holds for the \emph{first}
exact eigenvalue.  We call such a family of graph-like manifolds
\emph{uniform}.

Let us now specify $\eps_i$ and $\ell_i$ in dependence of the number
of vertices $\nu_i$ of $G^i$. We assume that
\newcommand{\myalpha}{\alpha} 
\newcommand{\mybeta}{\beta} 
\begin{equation}
  \label{eq:eps.ell.nu}
  \eps_i = \nu_i^{-\myalpha}
  \qquad\text{and}\qquad
  \ell_i = \nu_i^{-\mybeta}
\end{equation}
for some $\myalpha>0$ and $\mybeta \in \R$ (negative values for
$\mybeta$ are not excluded).  In particular, $X_0^i$ now also depends
on $\eps$, and we write $\nu_i^{-\mybeta} X_0^i$ for the metric graph
$X_0^i$ with all edge lengths being $\nu_i^{-\mybeta}$.
\begin{itemize}
\item
  \begin{itemize}
  \item For the $0$-form eigenvalue convergence in~\eqref{eq:ev.fcts}
    to hold we need $\max\{\mybeta,0\} < \myalpha/2$, as the error
    term is of order $\eps_i^{1/2}/\min\{\ell_i,1\}=\nu_i^{-\myalpha/2
      + \max\{\mybeta,0\}}$.

  \item For the metric graph eigenvalue, we have
    $\lambda_j(\nu_i^{-\mybeta} X_0^i)=\nu_i^{2\mybeta} \lambda_j(X_0^i)$.
  \item For the metric graph eigenvalue (of order $\nu_i^{2\mybeta}$)
    to be dominant with respect to the error (of order
    $\nu_i^{-\myalpha/2 + \max\{\mybeta,0\}}$), we need $\mybeta \ge -\myalpha/2$,
    $\mybeta \ge 0$ or $\mybeta \ge -\myalpha/4$, $\mybeta \le 0$.  Hence
    \begin{equation}
      \label{eq:ev.fcts''}
      \tag{\ref{eq:ev.fcts}''}  
      \lambda_j^0(X_\eps^i)
      \begin{cases}
        \eqsim \nu_i^{2\mybeta} \lambda_j(X_0^i),& 
                               (\mybeta \ge -\myalpha/2, \mybeta \ge 0) 
                                   \text{ or }
                               (\mybeta \ge -\myalpha/4, \mybeta \le 0),\\
        \lesssim \nu_i^{-\myalpha/2}, & \text{otherwise.}
      \end{cases}
    \end{equation}
  \end{itemize}
\begin{figure}[h]
  \centering
  \input{alpha-beta-diagram.pstex_t}
  \caption{
    \label{fig:alpha-beta}
    From left to right: Parameter regions\newline %
    $\bullet$~where the $0$-form eigenvalue convergence
    in~\eqref{eq:ev.fcts} holds ($\max\{\mybeta,0\} < \myalpha/2$)
    \newline %
    $\bullet$~where $\lambda_j^0(X_{\eps_i}^i) \eqsim \nu_i^{2\mybeta}
    \lambda_j(X_0^i)$ ($\mybeta > -\myalpha/2, \mybeta \ge 0$ or
    $\mybeta > -\myalpha/4, \mybeta \le 0$) \newline %
    $\bullet$~where $\exactEV 1 p(X_{\eps_i}^i)$ diverges. ($\alpha >1/2,
    \alpha \ge \beta$ or $4\alpha-2\beta-1>0$, $\alpha \le \beta$).
    \newline %
    \emph{Most left, dark grey:} region where all eigenvalues
    diverge. \emph{Light grey:} region where form eigenvalues diverge,
    function eigenvalues converge to $0$. \emph{Dotted line:} volume
    is constant; above: volume tends to $0$; below: volume tends to
    $\infty$.  \newline %
    \emph{Second row:} Regions of divergence of rescaled $0$-form
    eigenvalue: very dark grey for $n=4$ and very dark, darker grey
    for $n=3$ and $n=2$; dashed lines are $\beta=\alpha-1/(n-1)$ for
    $n \in \{2,3,4\}$.}
\end{figure}
  
\item For the divergence in~\eqref{eq:ev.forms} to hold we need
  $\myalpha>1/2$ (resp.\ $2\myalpha>1+\mybeta$).  In particular, we
  have
  \begin{equation}
    \label{eq:ev.forms''}
    \tag{\ref{eq:ev.forms}''}
    \exactEV 1 p(X_{\eps_i}^i) \gtrsim
    \begin{cases}
      \nu_i^{2\myalpha-1},& \myalpha \ge \mybeta,\\
      \nu_i^{4\myalpha-2\mybeta-1},& \myalpha \le \mybeta
    \end{cases}
  \end{equation}
  for $2 \le p \le n-1$.
\item For the volume, we have
  \begin{equation}
    \label{eq:volume''}
    \tag{\ref{eq:volume}''}
    \vol X_\eps \eqsim \nu_i^{-n\myalpha + 1} +\nu_i^{-(n-1)\myalpha-\mybeta+1}
    \eqsim
    \begin{cases}
        \nu_i^{-(n-1)\myalpha-\mybeta+1},&  \myalpha \ge \mybeta,\\
        \nu_i^{-n\myalpha + 1}, & \myalpha \le \mybeta.
      \end{cases}
  \end{equation}
\end{itemize}

\paragraph{Ramanujan graphs:}
Let us now conclude from the above and the diagram some examples. We
assume that the underlying sequence of metric graphs $X_0^i$ (with all
lengths equal to $1$) is Ramanujan, i.e., the (metric) graph
Laplacians have a common spectral gap $(0,h)$.  If we choose
$(\alpha,\beta)$ from the dark grey area ($\alpha > 1/2$, $\beta \ge
0$, $\beta < \alpha/2$) we have:
\begin{prop}
  \label{prp:ramanujan.ex1}
  There is a sequence of graph-like manifolds $(X_{\eps_i}^i)_i$ with
  underlying Ramanujan graphs $G^i$ with $\nu_i=\abs{V(G^i)}$ many
  vertices, such that the Hodge Laplacian of all degrees has an
  arbitrarily large spectral gap, i.e., there exists $h_i \eqsim
  \nu_i^{\min\{2\mybeta,2\myalpha-1\}}\to \infty$ such that
  \begin{equation*}
    \spec{\Delta_{X_{\eps_i}^i}^\bullet} \cap (0, h_i) = \emptyset
  \end{equation*}
  and such that the volume shrinks to $0$, more precisely, $\vol
  X_{\eps_i}^i \eqsim \nu_i^{-(n-1)\myalpha-\mybeta+1}$.

  In particular, if $\mybeta=0$, then there exists a common spectral
  gap $(0,h)$ of the Hodge Laplacian.  If, additionally, $n=3$, then
  the volume decay can be made arbitarily small as $\alpha \searrow 1/2$,
  i.e., of order $\nu_i^{-2\alpha+1}$.
\end{prop}
\begin{proof}
  The proof follows from the above considerations of eigenvalue
  asymptotics.  Note that for a sequence of Ramanujan graphs, there
  exists $h>0$ such that the first non-zero eigenvalue of the metric
  graph Laplacian with unit edge length fulfils $\lambda_1(X_0^i) \ge
  h$ for all $i$, hence we can conclude divergence from the first line
  of~\eqref{eq:ev.fcts''}.
\end{proof}
Note that the length scale of the underlying metric graphs is of order
$\nu_i^{-\beta}$, but the radius is of order $\eps_i=\nu_i^{-\alpha}$,
which is smaller; hence the injectivity radius of $X_{\eps_i}^i$ is
of order $\eps_i=\nu_i^{-\alpha}$, and the curvature is of order
$\eps_i^{-2}=\nu_i^{2\alpha}$. 

\paragraph{Rescaling the metric:}
Let us now rescale the metric to have fixed volume (i.e., set $\wt g_i
:= (\vol (X_{\eps_i}^i, g_{\eps_i}))^{-2/n} g_{\eps_i}$) and consider
$\wt X^i:=(X_{\eps_i}^i,\wt g_i)$).  Then the latter manifold has
volume $1$.  Unfortunately, we cannot have divergence at all degrees
at the same time; e.g.\ for $n=3$ the conditions are
$\beta>\alpha-1/2$ for divergence of the eigenvalues of degree $0$,
while $\beta<\alpha-1/2$ is needed for divergence of exact $2$-forms.
But we can have divergence of $0$-forms and higher degree forms
separately.
\begin{figure}[h]
  \centering
  \input{alpha-beta-diagram0.pstex_t}
  \caption{
    \label{fig:alpha-beta0}
    From left to right: Parameter regions for rescaled metric (volume
    is $1$)\newline %
    $\bullet$~where the $0$-form eigenvalue $\lambda_1^0(\wt
    X_{\eps_i}^i)$ diverges; \newline %
    $\bullet$~where $\exactEV 1 p(\wt X_{\eps_i}^i)$ diverges ($2 \le
    p \le n-1$).}
\end{figure}
\begin{cor}
  \label{cor:div.fcts}
  For all $n\ge 2$ there exists a family of graph-like manifolds $\wt
  X^i$ of volume $1$ with underlying Ramanujan graphs such that the
  first non-zero eigenvalue on functions ($0$-forms) diverges.
\end{cor}
\begin{proof}
  The rescaling factor $\tau_i = (\vol (X_{\eps_i}^i,
  g_{\eps_i}))^{-1/n}$ is of order
  $\nu_i^{(1-1/n)\alpha+\beta/n-1/n}$.  The rescaled eigenvalue on
  functions fulfils
  \begin{equation}
    \label{eq:scale.fct}
    \lambda_1(\wt X^i)
    =
    \tau_i^{-2}\lambda_1(X_{\eps_i}^i)
    \eqsim \tau_i^{-2}\nu_i^{2\beta}\lambda_1(X_0^i)
    \eqsim \nu_i^{2/n-2(1-1/n)(\alpha-\beta)}\lambda_1(X_0^i)
  \end{equation}
  and the latter exponent is positive if and only if $\beta > \alpha-1/(n-1)$.
  The allowed parameters $(\alpha,\beta)$ lie inside the triangle
  $(0,0)$, $(4,-1)/(5(n-1))$, $(2,1)/(n-1)$ such that $\lambda_1(\wt
  X^i) \eqsim \nu_i^{2/n-\delta}$ (see the differently grey coloured
  regions in Figure~\ref{fig:alpha-beta0} (right) for different $n$).
  The difference $\beta-\alpha$ approaches its maximum on this
  triangle at the vertex $(0,0)$.  Hence for any $\delta>0$ there
  exists $(\alpha,\beta)$ inside the triangle such that $\lambda_1(\wt
  X^i) \eqsim \nu_i^{2/n-\delta}$.
\end{proof}
In particular, for $n=2$ we have:
\begin{cor}
  \label{cor:div.fcts.n=2}
  There exists a sequence of graph-like surfaces $\wt X^i$ of area $1$
  and genus $\gamma(\wt X^i)$ with underlying Ramanujan graphs such
  that the first non-zero eigenvalue on functions diverges.  Moreover,
  for any $\delta>0$ there exists a sequence $(\wt X^i)_i$ such that
  \begin{equation*}
    \lambda_1(\wt X^i)
    \eqsim \gamma(\wt X^i)^{1-\delta},
  \end{equation*}
  i.e., the bound in~\eqref{eq:kappa.n=2} is asymptotically almost
  optimal.
\end{cor}
\begin{proof}
  We have to choose $Y_e=\Sphere^1$ here, moreover we let the vertex
  neighbourhood be a sphere with $k$ discs removed (as in
  Example~\ref{ex:gl-mfd}).  In this case, the genus of the surface
  $\wt X^i$ is given by $1-\chi(G^i)$ where $\chi(G^i)$ is the Euler
  characteristic of the graph $G^i$, and hence
  \begin{equation*}
    \gamma(\wt X^i)
    = 1 - \abs{V(G^i)}+\abs{E(G^i)}
    = 1 - \nu_i + \frac k2 \nu _i
    = 1+\Bigl(\frac k2 - 1\Bigr) \nu_i \to \infty
  \end{equation*}
  as $i \to \infty$ as $k \ge 3$ for a Ramanujan graph.  In
  particular, $\gamma(\wt X^i) \eqsim \nu_i$.
\end{proof}

\paragraph{Arbitrarily large differential form spectrum, constant
  volume and arbitrary graphs:}
Let us now assume that $(G^i)$ is any sequence of graphs with
$\nu_i=\abs{V(G^i)}\to \infty$ as $i \to \infty$ and with degrees
bounded by $k$. As we want the form spectrum to diverge, we do not
need that the underlying graphs are Ramanujan.
\begin{prop}
  \label{prp:div.forms}
  For all $n\ge 3$ there exists a family of graph-like manifolds $\wt
  X^i$ of volume $1$ such that the first eigenvalue on exact $p$-forms
  diverges ($2 \le p \le n-1$).  Moreover, the first non-zero eigenvalue on
  functions converges.
\end{prop}
\begin{proof}
  The rescaled eigenvalue on $p$-forms fulfils
  \begin{equation*}
    \exactEV 1 p (\wt X^i)
    =
    \tau_i^{-2} \exactEV 1 p (X_{\eps_i}^i)
    \gtrsim \tau_i^{-2}\nu_i^{2\alpha-1}
    \eqsim \nu_i^{2(\alpha-\beta+1)/n-1}
  \end{equation*}
  (as $\alpha \ge \beta$, see~\eqref{eq:ev.forms''}) and the latter
  exponent is positive if and only if $\beta < \alpha-(n/2-1)$.  The allowed
  parameters $(\alpha,\beta)$ lie below this line (see
  Figure~\ref{fig:alpha-beta0} (left).

  For the first non-zero eigenvalue on functions, note first that
  $\lambda_1(X_0^i)$ (the first non-zero eigenvalue of the unilateral
  metric graph $X_0^i$) can be bounded from above by $\pi^2$, this
  follows immediately from the spectral relation~\eqref{spectra_rel}.
  Therefore, we conclude from~\eqref{eq:scale.fct} that $\lambda_1(\wt
  X^i) \to 0$ as $i \to \infty$ as $\beta<\alpha-(n/2-1)$ implies that
  $2/n-2(1-n)(\alpha-\beta)<0$.
\end{proof}
Actually, comparing the speed of divergence and convergence, we obtain
\begin{equation*}
  \exactEV 1 p (\wt X^i) \gtrsim
  \nu_i^{\frac{n^2}{2(n-1)}} \lambda_1(\wt X^i)^{-\frac{n^2}{4(n-1)}},
\end{equation*}
confirming again that we cannot have divergence for both eigenvalues
with our construction.

If we choose the family of graphs $(G^i)_i$ to consist of trees only,
we can modify any given manifold $X$ to become a graph-like manifold
with underlying tree graph (``growing a tree on $X$'', see
Remark~\ref{rem:gl-mfd}).  In particular, we can show the following.
\begin{cor}
  \label{cor:any.mfd}
  On any compact manifold $X$ of dimension $n \ge 3$, there exists a
  sequence of metrics $g_i$ of volume $1$ such that the infimum of the
  (non-zero) function spectrum converges to $0$, while the exact
  $p$-form eigenvalues ($2\le p \le n-1$) diverge.
\end{cor}

%
%
\addcontentsline{toc}{section}{References}
\ifthenelse{\isundefined \Olaf }
{\ifthenelse{\isundefined \otherauthorusesbibtex}
  {} 
  {
    \bibliographystyle{amsalpha}  
    \bibliography{\bibtexpath}
  }
}
{ 
  \bibliographystyle{/home/post/Aktuell/BibTeX/my-amsalpha}
  \bibliography{/home/post/Aktuell/BibTeX/literatur}
}

\providecommand{\bysame}{\leavevmode\hbox to3em{\hrulefill}\thinspace}

\end{document}